\documentclass[12pt,a4paper,reqno]{amsart}
\usepackage{amsmath,amsfonts,amscd,amssymb,latexsym,mathtools}
\usepackage[left=1.2in,right=1.2in,top=1in,bottom=1in]{geometry}

\usepackage{hyperref}
\usepackage{mathrsfs}
\usepackage[percent]{overpic} 

\usepackage[nobysame]{amsrefs}
\usepackage{comment}
\usepackage{soul}
\usepackage{tikz}
\usepackage{marginnote}
\usepackage{bbm}
\usepackage{cancel}
\usepackage[normalem]{ulem}
\usepackage{txfonts}

\numberwithin{equation}{section}

\newtheorem{theorem}{Theorem}[section]
\newtheorem{lemma}[theorem]{Lemma}
\newtheorem{proposition}[theorem]{Proposition}
\newtheorem{corollary}[theorem]{Corollary}

\theoremstyle{definition}
\newtheorem{definition}[theorem]{Definition}
\newtheorem{example}[theorem]{Example}
\newtheorem{assumption}[theorem]{Assumption}

\theoremstyle{remark}
\newtheorem{remark}[theorem]{Remark}

\newcommand{\ind}{\operatorname{ind}}
\newcommand{\id}{\operatorname{id}}
\newcommand{\dom}{\operatorname{dom}}

\newcommand{\CC}{\mathbb{C}}

\newcommand{\RR}{\mathbb{R}}

\newcommand{\ZZ}{\mathbb{Z}}

\newcommand{\AAA}{\mathcal{A}}

\newcommand{\BBB}{\mathcal{B}}
\newcommand{\DD}{\mathcal{D}}
\newcommand{\EE}{\mathcal{E}}

\newcommand{\HH}{\mathcal{H}}
\newcommand{\II}{\mathcal{I}}
\newcommand{\KK}{\mathcal{K}}
\newcommand{\MM}{\mathcal{M}}

\newcommand{\aps}{{\rm APS}}
\newcommand{\daps}{{\rm dAPS}}
\newcommand{\rmc}{\mathrm{c}}

\newcommand{\upper}{\uppercase\expandafter}

\newcommand{\p}{\partial}

\newcommand{\supp}{\operatorname{supp}}
\newcommand{\Hom}{\operatorname{Hom}}

\newcommand{\Res}{\operatorname{Res}}

\newcommand{\tr}{\operatorname{tr}}
\newcommand{\Tr}{\operatorname{Tr}}

\newcommand{\spf}{\operatorname{sf}}

\renewcommand{\AA}{\mathbb{A}}

\newcommand{\tilKK}{\tilde{\mathcal{K}}}
\newcommand{\AS}{\alpha_{\rm AS}}
\newcommand{\oeta}{\bar{\eta}}
\newcommand{\oxi}{\bar{\xi}}

\newcommand{\tilA}{\tilde \AAA}
\newcommand{\tilB}{\tilde \BBB}

\newcommand{\tilPi}{\tilde \Pi}

\begin{document}

\normalsize

\title[Relative eta invariant]{The relative eta invariant for a pair of Dirac-type operators on non-compact manifolds}

\author[Pengshuai Shi]{Pengshuai Shi${}^\ast$}
\address{School of Mathematics and Statistics, Beijing Institute of Technology, Beijing 100081, China}

\email{pengshuai.shi@gmail.com}



\thanks{${}^\ast$Partially supported by Beijing Institute of Technology Research Fund Program for Young Scholars}

\begin{abstract}
Let $\AAA_0$ and $\AAA_1$ be two self-adjoint Fredholm Dirac-type operators defined on two non-compact manifolds. If they coincide at infinity so that the relative heat operator is trace-class, one can define their relative eta function as in the compact case. The regular value of this function at the zero point, which we call the relative eta invariant of $\AAA_0$ and $\AAA_1$, is a generalization of the eta invariant to non-compact situation. We study its variation formula and gluing law. In particular, under certain conditions, we show that this relative eta invariant coincides with the relative eta invariant that we previously defined using APS index of strongly Callias-type operators.
\end{abstract}

\maketitle

\setcounter{tocdepth}{1}
\tableofcontents

\section{Introduction}\label{S:intro}

The eta invariant is a non-local spectral invariant that was first introduced by Atiyah, Patodi and Singer in \cite{APS1} as a boundary contribution of an index formula on manifolds with boundary under APS boundary condition. It can be defined as follows. Let $M$ be a compact Riemannian manifold without boundary, and $D$ be a self-adjoint first-order elliptic differential operator acting on sections of a Hermitian vector bundle $E$ over $M$. It is well-known that $D$ has \emph{pure discrete spectrum} consisting of real eigenvalues with finite multiplicity. Define the eta function of $D$ by
\begin{equation}\label{E:etafunc}
\eta(s;D)\;:=\;\sum_{\lambda\in{\rm spec}(D)\setminus\{0\}}{\rm sign}(\lambda)|\lambda|^{-s},\qquad\text{for }\Re(s)>\dim M.
\end{equation}
It is absolutely convergent in the half-plane $\Re(s)>\dim M$ since the eigenvalues of $D$ satisfy a Weyl's law. Using Mellin transform, the eta function can also be written in terms of the heat operator
\begin{equation}\label{E:etafunc-2}
\eta(s;D)\;=\;\frac{1}{\Gamma((s+1)/2)}\,\int_0^\infty t^{(s-1)/2}\Tr De^{-tD^2}\,dt.
\end{equation}
It follows from this expression that the eta function admits a meromorphic continuation to the whole complex plane. Moreover, $s=0$ is a regular point of $\eta(s;D)$. This was shown in \cites{APS3,Gilkey81globalEta} for general elliptic pseudo-differential operators. Thus the eta invariant of $D$ is defined to be $\eta(D):=\eta(0;D)$.

Since \cites{APS1,APS2,APS3}, the eta invariant has been studied extensively and generalized to various situations, for example, manifolds with boundary \cites{GilkeySmith83,DouglasWoj91}, manifolds with conical singularities \cites{Cheeger83singular,Cheeger87conical,BismutCheeger91remarks}, manifolds with cylindrical ends \cites{KlimekWoj93,Muller94}; and some recent developments including the case of convex co-compact hyperbolic manifolds \cite{GMP10}, manifolds with cusps \cite{LMP11cusps}, manifolds with periodic ends \cite{MRS16}, manifolds with edges \cite{PiazzaVertman19edge}, etc. It turns out that the eta invariant plays more and more important roles in geometry and topology. For more about this, the reader can consult Goette \cite{Goette12compEta} and the references therein.

While it makes perfect sense to talk about eta invariant in many different settings, a general extension to non-compact manifolds is not expected, mainly because of the presence of continuous spectrum. In fact, even if an operator on a non-compact manifold has discrete spectrum, its eta function can still be undefined. However, like Gromov--Lawson's Relative Index Theorem \cite{GromovLawson83}, one can consider a ``relative version'' of the eta invariant. That is, consider two operators on two non-compact manifolds. If the manifolds and the corresponding operators coincide outside compact subsets, then one can define a relative eta function in terms of the relative heat trace of the two operators. Under some conditions, this function may be eligible to produce a relative eta invariant. This question was studied abstractly by M\"uller in \cite{Muller98} under certain assumptions on the asymptotic expansions of relative heat traces. In this paper, we restrict our discussion to Dirac-type operators, as they are a central class of first-order elliptic differential operators of geometric origin and have a wide range of applications. We show that the conditions of \cite{Muller98} can be satisfied in our setting such that a relative eta invariant would exist. In other words, the notion of relative eta invariant is validly established under very general circumstances in the present article. Apart from that, we then study its properties.

Our setting is as follows. Let $\AAA_0$ and $\AAA_1$ be two self-adjoint Dirac-type operators on non-compact manifolds $Y_0$ and $Y_1$, respectively. Suppose $\AAA_0$ and $\AAA_1$ coincide at infinity. Then under some assumptions (Assumption \ref{A:traceclass}), Bunke \cite{Bunke92} showed that $\AAA_1e^{-t\AAA_1^2}-\AAA_0e^{-t\AAA_0^2}$ is a trace-class operator using techniques of heat kernel estimates.\footnote{In \cite{EichhornBook09}, Eichhorn investigated the question of whether such an operator is of trace class in a more general setting and then studied the relative spectral invariants on open manifolds including a brief discussion about relative eta invariant.} Along the line we get that the short time asymptotic expansion of the trace is determined by the corresponding expansions on closed manifolds. If $\AAA_0$ and $\AAA_1$ are Fredholm operators, then the above trace decays exponentially for large time. This enables us to make sense of the following \emph{relative eta function}
\begin{equation}\label{E:releta func}
		\eta(s;\AAA_1,\AAA_0) \;:= \;
	\frac{1}{\Gamma((s+1)/2)}\,\int_0^\infty
	  t^{(s-1)/2}
	     \Tr\big(\, \AAA_1e^{-t\AAA_1^2}-\AAA_0e^{-t\AAA_0^2}\,\big) \,dt
\end{equation}
for $s$ with large real part. In this case, the relative eta function still admits a meromorphic continuation to the whole complex plane. We show the regularity of $\eta(s;\AAA_1,\AAA_0)$ at $s=0$.

\begin{theorem}\label{T:intro-1}
The relative eta function $\eta(s;\AAA_1,\AAA_0)$ is regular at $s=0$.
\end{theorem}

\noindent From this, one can define $\eta(0;\AAA_1,\AAA_0)$ as the \emph{relative eta invariant} of $\AAA_0$ and $\AAA_1$.

One of the most important features of the eta invariant is its role in the APS index theorem. If there is a relative eta invariant, one would expect it to be a part of an APS index formula for manifolds with \emph{non-compact boundary}. In fact, in \cites{BrShi17,BrShi17-2} with M. Braverman, we have studied the APS index problem for strongly Callias-type operators on manifolds with non-compact boundary and reduced the problem to a model space called essentially cylindrical manifolds, whose boundary consists of exactly two non-compact pieces which coincide outside compact subsets. In this situation we got a boundary term in the formula, which we also called relative eta invariant. This term is denoted by $\eta(\AAA_1,\AAA_0)$.

It was conjectured in \cites{BrShi17} that the boundary term in the APS index formula should indeed be the invariant defined from \eqref{E:releta func}. This problem is the main concern of the current paper and is answered partially in Sections \ref{S:varFormula} and \ref{S:gluing law}. To make it clear, we call the one defined from APS index \emph{index-theoretic} relative eta invariant and the one defined from \eqref{E:releta func} \emph{spectral} relative eta invariant. Then the main result can be formulated as

\begin{theorem}\label{T:intro-2}
Let $\AAA_0$ and $\AAA_1$ be two cobordant self-adjoint strongly Callias-type operators over two non-compact manifolds.
\begin{enumerate}
\item If $\AAA_0$ and $\AAA_1$ act on the same domain and satisfy Assumption \ref{A:traceclass}, then $\eta(\AAA_1,\AAA_0)=\eta(0;\AAA_1,\AAA_0)$.
\item If $\AAA_0$ and $\AAA_1$ act on different domains and satisfy Assumption \ref{A:traceclass-bvp}, then $\eta(\AAA_1,\AAA_0)=\eta(0;\AAA_1,\AAA_0)\mod2\ZZ$.
\end{enumerate}
\end{theorem}

Ideally, the equality between the two invariants can be achieved by directly computing the APS index, as  in \cite{FoxHaskell05}. However this requires much stronger conditions. In our cases, Assumption \ref{A:traceclass} is a fairly weak condition which is necessary for the definition of spectral relative eta invariant. And even the stronger one, Assumption \ref{A:traceclass-bvp}, is weaker than the conditions to do direct computations. As a result we must employ a different approach to prove the theorem. The ideas are as follows. For part (\romannumeral1), we consider a linear path linking $\AAA_0$ to $\AAA_1$. Then both invariants vary continuously along the path modulo integer jumps. We computed the variation formulas of the two respective invariants, which turn out to be equal. On the other hand, the integer jumps for both invariants can be described by the spectral flow. Therefore (\romannumeral1) is proved. The main ingredient in proving part (\romannumeral2) is the gluing law of the relative eta invariants. In this process we follow the treatment of Br\"uning and Lesch \cite{BruningLesch99} on compact manifolds and adapt it to the non-compact situation. Compared to the proof of part (\romannumeral1), this method works for general cases but only in the sense of mod $2\ZZ$ at this stage.

The paper is organized as follows. Section \ref{S:ind-eta} provides a brief review on the index-theoretic definition of relative eta invariant in \cites{BrShi17,BrShi17-2}. In Section \ref{S:relheattrace} we introduce our setting and study the short and large time asymptotic behaviors of the relative heat traces. Using the results of Section \ref{S:relheattrace}, we establish the definition of spectral relative eta invariant in Section \ref{S:releta}, stating the regularity (Theorem \ref{T:intro-1}) in Theorem \ref{T:releta func}. In Section \ref{S:varFormula}, the variation formulas for the index-theoretic and spectral relative eta invariants are computed respectively in Theorems \ref{T:varFormula-noncpt-ind} and \ref{T:varFormula-noncpt-sp}, among which the former theorem is less well-known. Together with their relationships with the spectral flow, one derives part (\romannumeral1) of Theorem \ref{T:intro-2}, which is formulated as Theorem \ref{T:releta: ind=sp}. In Section \ref{S:gluing law} we obtain the gluing law for the relative eta invariants in Theorems \ref{T:gluing-ind} and \ref{T:gluing-sp} and combine them to show part (\romannumeral2) of Theorem \ref{T:intro-2}, which is formulated as Theorem \ref{T:releta: ind=sp mod2Z}.

\subsection*{Acknowledgements}

The author would like to thank Prof. Maxim Braverman for helpful discussions and comments and Prof. Gang Tian for encouragements and supports. The author is partially supported by Beijing Institute of Technology Research Fund Program for Young Scholars (No. 3170012222013).

\section{The index theoretic definition of relative eta invariant}\label{S:ind-eta}

We briefly review the results in \cites{BrShi17,BrShi17-2} with M. Braverman about the index-theoretic definition of the relative eta invariant and its properties.

\subsection{The APS index for strongly Callias-type operators and the relative eta invariant}\label{SS:APSind}

Let $X$ be an $(n+1)$-dimensional complete Riemannian manifold\footnote{Here complete means complete as a metric space.}
and $E$ be a Dirac bundle over $X$. Basically, $E$ is a Hermitian vector bundle with a Clifford multiplication $\rmc(\cdot)$ of $TX$ and a connection $\nabla^E$ that is compatible with $\rmc(\cdot)$ (cf. \cite[\S II.5]{LawMic89}). The (compatible) Dirac operator $D:C^\infty(X,E)\to C^\infty(X,E)$ is then defined by
\[
C^\infty(X,E)\,\xrightarrow{\nabla^E}\,C^\infty(X,T^*X\otimes E)\,\rightarrow\, C^\infty(X,TX\otimes E)\,\xrightarrow{\rmc(\cdot)}\,C^\infty(X,E),
\]
where the second arrow is identifying $TX$ and $T^*X$ by the Riemannian metric. It is well known that $D$ is a formally self-adjoint operator. We will consider $E=E^+\oplus E^-$ to be $\ZZ_2$-graded so that $D=D^+\oplus D^-$, where $D^\pm:C^\infty(X,E^\pm)\to C^\infty(X,E^\mp)$.

Let $\Psi=\Psi^+\oplus\Psi^-\in\Hom(E^+\oplus E^-,E^-\oplus E^+)$ be a self-adjoint bundle map.
\[
\DD\;=\;D+\Psi\;=\;(D^++\Psi^+)\oplus(D^-+\Psi^-)\;=:\;\DD^+\oplus\DD^-
\]
is called a \emph{strongly Callias-type operator}, if $\DD^2=D^2+V$, where $V=[D,\Psi]_++\Psi^2$ is a bundle map\footnote{Here $[\cdot,\cdot]_+$ denotes anti-commutator and the condition means that $\Psi$ anti-commutes with the Clifford multiplication.} such that for any $R>0$, there exists a compact set $K_R\Subset X$ with $V(x)\ge R$ for all $x\in X\setminus K_R$. $\Psi$ is called a \emph{Callias potential} and the above $K_R$ is called an \emph{essential support} of $\DD$. $\DD^+$ and $\DD^-$ are formal adjoint to each other.

Suppose $X$ has non-compact boundary and $\DD$ is a product near the boundary $\p X=Y$, i.e.,
\[
\begin{aligned}
\DD^+&\;=\;\rmc(\nu)(\p_u+\AAA), \\
\DD^-&\;=\;\rmc(\nu)(\p_u+\AAA^\sharp),
\end{aligned}
\]
on a neighborhood $Z_r=[0,r)\times Y$ of $Y$. Here $\nu$ is the inward unit normal vector, $\AAA=A-\rmc(\nu)\Psi^+:C^\infty(Y,E^+|_Y)\to C^\infty(Y,E^+|_Y)$ is a self-adjoint strongly Callias-type operator and $\AAA^\sharp=\rmc(\nu)\circ\AAA\circ\rmc(\nu):C^\infty(Y,E^-|_Y)\to C^\infty(Y,E^-|_Y)$, where $A$ is the Dirac operator on $(Y,E^+|_Y)$. We call $\AAA$ (resp. $\AAA^\sharp$) the \emph{restriction} of $\DD^+$ (resp. $\DD^-$) to $Y$.

Note that $\AAA$ and $\AAA^\sharp$ have discrete spectra. Using $\AAA$ and $\AAA^\sharp$, one can define the Sobolev spaces $H_I^s(\AAA)$ and $H_I^s(\AAA^\sharp)$ on $Y$ for any $s\in\RR$ and $I\subset\RR$. Then we can consider the \emph{Atiyah--Patodi--Singer boundary value problem} $\DD_\aps^+$, which is $\DD^+$ with APS boundary condition $H_{(-\infty,0)}^{1/2}(\AAA)$. It is proved in \cite{BrShi17} that $\DD_\aps^+$ is Fredholm and its index is defined to be
\[
\ind\DD_\aps^+\;:=\;\dim\ker\DD_\aps^+\;-\;\dim\ker\DD_\daps^-\;\in\;\ZZ,
\]
where $\DD_\daps^-$ is the operator $\DD^-$ with dual APS boundary condition $H_{(-\infty,0]}^{1/2}(\AAA^\sharp)$. Furthermore, this index problem can be reduced to a so-called \emph{essentially cylindrical manifold} which contains an essential support of $\DD^+$.

Therefore, we will just assume that $X$ is already an essentially cylindrical manifold henceforth. This means that the boundary $Y$ is a disjoint union of two components $Y=Y_0\sqcup Y_1$ such that
\begin{enumerate}
\item
there exist a compact set $K\subset X$, an open manifold $Y'$, and an isometry $X\setminus K\simeq [0,\varepsilon]\times Y'$;

\item
under the above isometry $Y_0\setminus K =\{0\}\times Y'$ and $Y_1\setminus K =\{\varepsilon\}\times Y'$. 
\end{enumerate}
We also assume that $E$ and $\DD$ are in product form on $X\setminus K$. In this case, $X$ is called an \emph{almost compact essential support} of $\DD$.

Let $\AAA_0$ (resp. $-\AAA_1$) be the restriction of $\DD^+$ to $Y_0$ (resp. $Y_1$). One considers the APS boundary value problem $\DD_{B_0\oplus B_1}^+$, where $B_0=H_{(-\infty,0)}^{1/2}(\AAA_0)$ and $B_1=H_{(-\infty,0)}^{1/2}(-\AAA_1)=H_{(0,\infty)}^{1/2}(\AAA_1)$. Let $\alpha_{\rm AS}(\DD^+)$ be the Atiyah--Singer integrand of $\DD^+$. By the product structures of $X$ and $E$ outside $K$, the top degree component of $\AS(\DD^+)$ vanishes on $X\setminus K$. Hence the integral $\int_X\alpha_{\rm AS}(\DD^+)$ is well-defined and finite. In \cite[Theorem 3.4]{BrShi17-2}, we show that $\ind \DD^+_{B_0\oplus B_1}-\int_X\alpha_{\rm AS}(\DD^+)$ is a quantity depending only on the restrictions $\AAA_0$ and $\AAA_1$. This makes it possible to define a relative eta invariant.

\begin{definition}\label{D:almost compact cobordism}
An {\em almost compact cobordism} between $\AAA_0$ and $\AAA_1$ is a pair $(X,\DD)$, where $X$ is an essentially cylindrical manifold with $\p X=Y_0\sqcup Y_1$ and $\DD$ is a $\ZZ_2$-graded self-adjoint strongly Callias-type operator on $X$ such that 
\begin{enumerate}
\item $X$ is an almost compact essential support of $\DD$ so that $\DD$ is product near $\p X$;
\item The restriction of $\DD^+$ to $Y_0$ is equal to $\AAA_0$ and the restriction of $\DD^+$ to $Y_1$ is equal to $-\AAA_1$.
\end{enumerate}
If there exists an almost compact cobordism between $\AAA_0$ and $\AAA_1$ we say that operator $\AAA_0$ is {\em cobordant} to operator $\AAA_1$.
\end{definition}

\begin{definition}\label{D:releta}
Suppose $\AAA_0$ and $\AAA_1$ are cobordant self-adjoint strongly Callias-type operators and let $(X,\DD)$ be an almost compact cobordism between them. Let $\DD_\aps^+$ be $\DD^+$ with the APS boundary conditions $B_0= H_{(-\infty,0)}^{1/2}(\AAA_0)$ and $B_1= H_{(-\infty,0)}^{1/2}(-\AAA_1)$. The (index-theoretic) {\em relative eta invariant} is defined to be 
\[
	\eta(\AAA_1,\AAA_0) \;:= \; 2\,\Big(\ind \DD^+_\aps
	\,-\,\int_X\,\AS(\DD^+)\Big)
	 \,+ \,\dim\ker \AAA_0\,+ \,\dim\ker \AAA_1\;\in\;\RR.
\]
\end{definition}
Note that  $\eta(\AAA_1,\AAA_0)$ is independent of the choice of the cobordism $(X,\DD)$. Sometimes we are interested in the \emph{reduced relative eta invariant}
\begin{equation}\label{E:reduced releta-1}
\xi(\AAA_1,\AAA_0)\;:=\;\frac{1}{2}\,(\eta(\AAA_1,\AAA_0)\,+\,\dim\ker\AAA_1\,-\,\dim\ker\AAA_0).
\end{equation}
Then by definition
\begin{equation}\label{E:reduced releta-2}
\xi(\AAA_1,\AAA_0)\;=\;\ind \DD^+_{B_0\oplus B_1}
	\,-\,\int_X\,\AS(\DD^+)
	 \,+ \,\dim\ker \AAA_1.
\end{equation}

\begin{remark}\label{R:releta}
When $\dim X$ is odd, $\AS(\DD^+)$ vanishes and $\eta(\AAA_1,\AAA_0)$, $\xi(\AAA_1,\AAA_0)$ are integers. Otherwise they are just real numbers.
\end{remark}

\begin{proposition}[\cites{BrShi17,BrShi17-2}]\label{P:reletainv-prop}
The relative eta invariant defined above satisfies the following properties
\begin{enumerate}
\item $\eta(\AAA_0,\AAA_0)=0$,
\item $\eta(\AAA_2,\AAA_1)+\eta(\AAA_1,\AAA_0)=\eta(\AAA_2,\AAA_0)$.
\end{enumerate}
\end{proposition}

\subsection{Relative eta invariant and the spectral flow}\label{SS:sf}

Suppose $\AA:= \{\AAA_r\}_{0\le r\le 1}$ is a smooth family of self-adjoint elliptic operators on a closed manifold. Let $\oeta(\AAA_r)\in \RR/\ZZ$ denote the mod $\ZZ$ reduction of the eta invariant $\eta(\AAA_r)$. Atiyah, Patodi and Singer \cite{APS3}, showed that $r\mapsto \oeta(\AAA_r)$ is a smooth function whose derivative $\frac{d}{dr}\oeta(\AAA_r)$ is given by an explicit local formula. 
Further, they introduced a notion of spectral flow $\spf(\AA)$ (which is roughly the net number of eigenvalues that change sign when $s$ changes from 0 to 1) and showed that it can be computed in terms of the eta invariant, i.e., 
\begin{equation}\label{E:sf-eta}
	2\spf(\AA)\;= \;
	\eta(\AAA_1)\;- \;\eta(\AAA_0)\;- \;
	\int_0^1\Big(\frac{d}{dr}\oeta(\AAA_r)\,\Big)\, dr. 
\end{equation}

In \cites{BrShi17,BrShi17-2}, we consider a family of self-adjoint strongly Callias-type operators $\AA= \{\AAA_r\}_{0\le r\le 1}$ on a complete \emph{non-compact} manifold. In this case one can still define the spectral flow $\spf(\AA)$ of the family $\AA$ and we get a formula similar to $\eqref{E:sf-eta}$.

\begin{theorem}[\cite{BrShi17-2}]\label{T:sp flow}
Let
\(\AA  =  \big\{
	  \AAA_r:\,C^\infty(Y_1,E_1)\to C^\infty(Y_1,E_1)\big\}_{0\le r\le 1}  
\)
be a smooth family of self-adjoint strongly Callias-type operators on a complete Riemannian manifold $Y_1$. Assume that $\AAA_r$ is constant in $r$ outside a compact subset of $Y_1$. Assume also that $\AAA_0$ and $\AAA_1$ are invertible. Let $\AAA^0:C^\infty(Y_0,E_0)\to C^\infty(Y_0,E_0)$ be an invertible self-adjoint strongly Callias-type operator on a complete Riemannian manifold $Y_0$ which is cobordant to the family $\AA$. Then the mod $\ZZ$ reduction $\oeta(\AAA_r,\AAA^0)\in \RR/\ZZ$ of the relative eta invariant depends smoothly on $r\in[0,1]$ and 
\begin{equation}\label{E:sp flow}
	\eta(\AAA_1,\AAA^0)\;- \;\eta(\AAA_0,\AAA^0)
	\;- \;\int_0^1\Big(\frac{d}{dr}\oeta(\AAA_r,\AAA^0)\Big)\,dr
	\;= \;2\,\spf(\AA).
\end{equation}
\end{theorem}

When $\dim Y_0=\dim Y_1$ is even, by Remark \ref{R:releta}, the integral term on the left hand side of \eqref{E:sp flow} vanishes.

Essentially, the proof of the theorem does not depend on the invertibility assumption. In fact, we have

\begin{corollary}\label{C:sp flow}
For $0\le r\le1$, let $\xi(\AAA_r,\AAA^0)$ be the reduced relative eta invariant defined in \eqref{E:reduced releta-1}. Then under the same hypothesis as in Theorem \ref{T:sp flow} except for the assumption that $\AAA_0,\AAA_1$ or $\AAA^0$ being invertible, one has
\begin{equation}\label{E:sp flow-1}
	\xi(\AAA_1,\AAA^0)\;- \;\xi(\AAA_0,\AAA^0)
	\;- \;\frac{1}{2}\int_0^1\,\Big(\frac{d}{dr}\oeta(\AAA_r,\AAA^0)\Big)\,dr
	\;= \;\spf(\AA).
\end{equation}
\end{corollary}

\subsection{A conjecture on the spectral interpretation of the relative eta invariant}\label{SS:conj-releta}

For the eta invariant on a compact manifold, one can see from \eqref{E:etafunc} or \eqref{E:etafunc-2} that it is a spectral invariant. From this perspective, the relationship \eqref{E:sf-eta} between it and the spectral flow becomes natural.

As mentioned in the Introduction, to give a spectral interpretation to the relative eta invariant defined above, one needs to generalize \eqref{E:etafunc-2} to \eqref{E:releta func} and relate it to $\eta(\AAA_1,\AAA_0)$. Therefore we conjectured in \cite{BrShi17} that if \eqref{E:releta func} is defined, analytic and regular at $s=0$, then $\eta(\AAA_1,\AAA_0)=\eta(0;\AAA_1,\AAA_0)$.

In the following sections, we will investigate the relative eta function and its regularity at zero for a larger class of operators and give a partial answer to this conjecture.

\section{Relative heat trace and its asymptotic properties}\label{S:relheattrace}

The purpose of this section is to study the short time asymptotic expansion and large time asymptotic behavior of the term $\Tr(\AAA_1e^{-t\AAA_1^2}-\AAA_0e^{-t\AAA_0^2})$ appearing in \eqref{E:releta func}. In next section we will use these properties to discuss the well-definedness of the relative eta function \eqref{E:releta func}. It turns out that for these discussions, the requirements for the operators can be weakened as follows.

\begin{definition}\label{D:condition at infinity}
For $j=0,1$, let $\AAA_j$ be a self-adjoint Dirac-type operator over a complete non-compact $n$-dimensional Riemannian manifold $Y_j$ without boundary and acting on sections of a Dirac bundle $E_j\to Y_j$. We say that $\AAA_0$ and $\AAA_1$ \emph{coincide and are invertible at infinity}\footnote{In this paper, we use ``at infinity'' to mean outside a compact set.},
if
\begin{enumerate}
\item they coincide at infinity, namely there exist compact subsets $K_0\Subset Y_0$ and $K_1\Subset Y_1$ and an isometry $\II:Y_0\setminus K_0\cong Y_1\setminus K_1$ which is covered by a bundle isometry $\tilde{\II}:E_0|_{Y_0\setminus K_0}\cong E_1|_{Y_1\setminus K_1}$ such that $\AAA_1=\tilde{\II}\circ\AAA_0\circ\tilde{\II}^{-1}$;

\item they are invertible at infinity, namely there exists a constant $C>0$ such that,
\[
\Vert\AAA_js_j\Vert_{L^2(Y_j,E_j)}\;\ge\;C\,\Vert s_j\Vert_{L^2(Y_j,E_j)},
\]
for any $s_j\in C_0^\infty(Y_j,E_j)$ whose support lies in $Y_j\setminus K_j$.
\end{enumerate}
\end{definition}

Throughout the paper, we will just identify $\AAA_0$ and $\AAA_1$ (i.e., $\II=\tilde{\II}=\id$) on the set $U:=Y_0\setminus K_0\cong Y_1\setminus K_1$ whenever there are no confusions. By \cite[Theorem 2.1]{Anghel93}, $\AAA_0$ and $\AAA_1$ are Fredholm operators. In particular, the essential spectra of $\AAA_0^2$ and $\AAA_1^2$ have a common positive lower bound.

\subsection{The trace class property}\label{SS:traceclass}

The first step is to show that $\AAA_1e^{-t\AAA_1^2}-\AAA_0e^{-t\AAA_0^2}$ is a trace class operator. Here we recall the result obtained by Bunke in \cite{Bunke92}. We regard both $\AAA_0$ and $\AAA_1$ as operators on $Y_0\cup_U Y_1:=K_0\sqcup K_1\sqcup U$ such that $\AAA_0$ is $0$ on $K_1$ and $\AAA_1$ is $0$ on $K_0$. Thus both operators act on the Hilbert space
\begin{equation}\label{E:space decomp}
\HH\;:=\;L^2(K_0,E_0|_{K_0})\,\oplus\,L^2(K_1,E_1|_{K_1})\,\oplus\,L^2(U,E_0|_U).
\end{equation}
Canonically, there is an orthogonal projection $P_j$ ($j=0,1$) from $\HH$ onto $L^2(Y_j,E_j)$. In this setting, the operator defined by $\AAA_j$ acts on the image of $P_j$ and acts as 0 on its orthogonal complement. 

In \cite{Bunke92}, Bunke proved that under certain conditions, $e^{-t\AAA_1^2}-e^{-t\AAA_0^2}$ and $\AAA_1e^{-t\AAA_1^2}-\AAA_0e^{-t\AAA_0^2}$ are trace-class operators. The technique used is heat kernel estimates. Although the operator $\AAA_j$ ($j=0,1$) discussed in \cite{Bunke92} is a generalized Dirac operator (without potential), the proof works for Dirac-type operators as well.

\begin{theorem}\label{T:traceclass}
Let $\AAA_0$ and $\AAA_1$ be two self-adjoint Dirac-type operators on $(Y_0,E_0)$ and $(Y_1,E_1)$, respectively which coincide at infinity. Assume that $Y_j$ has bounded sectional curvature, $\AAA_j^2-\nabla_j^*\nabla_j$ is a bundle endomorphism and that there is a lower bound for $\mathcal{R}=\AAA_0^2-\nabla_0^*\nabla_0=\AAA_1^2-\nabla_1^*\nabla_1$ on $U$, where $\nabla_j$ is the connection of the bundle $E_j$. Then $e^{-t\AAA_1^2}-e^{-t\AAA_0^2}$ and $\AAA_1e^{-t\AAA_1^2}-\AAA_0e^{-t\AAA_0^2}$ are trace-class operators for all $t>0$.
\end{theorem}

\begin{remark}\label{R:traceclass}
Note that the condition about $\mathcal{R}$ is different from the condition that $\AAA_j$ being invertible at infinity because the lower bound of $\mathcal{R}$ does not need to be positive. If $\AAA_j$ is a strongly Callias-type operator whose Callias potential is $\Psi_j$, i.e., $\AAA_j=A_j+\Psi_j$, then on $U$,
\[
\mathcal{R}\;=\;([A_j,\Psi_j]_+\,+\,\Psi_j^2)\;+\;(A_j^2\,-\,\nabla_j^*\nabla_j),\qquad j=0,1,
\]
where both terms on the right hand side are bundle maps and the lower bound assumption on $\mathcal{R}$ can be guaranteed by the rapid growth of $[A_j,\Psi_j]_++\Psi_j^2$ at infinity.
\end{remark}

This theorem is fundamental for the discussions below. Thus we make the following basic assumption.

\begin{assumption}\label{A:traceclass}
Let $\AAA_0$ and $\AAA_1$ be two self-adjoint Dirac-type operators which coincide and are invertible at infinity. Assume the two triples $(Y_0,E_0,\AAA_0)$ and $(Y_1,E_1,\AAA_1)$ satisfy the hypothesis of Theorem \ref{T:traceclass}.
\end{assumption}

\begin{corollary}\label{C:traceclass}
Let $\KK_j(y,z;t)$ be the kernel of the heat operator $e^{-t\AAA_j^2}$ ($j=0,1$). Under Assumption \ref{A:traceclass}, for any $t>0$,
\begin{align}
\Tr\big(e^{-t\AAA_1^2}-e^{-t\AAA_0^2}\big)&\;=\;\int_{Y_0\cup_U Y_1}\tr\big(\KK_1(y,y;t)-\KK_0(y,y;t)\big)\,dy, \label{E:traceclass-1} \\
\Tr\big(\AAA_1e^{-t\AAA_1^2}- \AAA_0e^{-t\AAA_0^2}\big)&\;=\;\int_{Y_0\cup_U Y_1}\tr\big(\AAA_1\KK_1(y,y;t)-\AAA_0\KK_0(y,y;t)\big)\,dy, \label{E:traceclass-2}
\end{align}
where $\tr$ denotes the pointwise trace and $\AAA_j$ acts with respect to the first spatial component.
\end{corollary}

\begin{remark}\label{R:rel trace indep}
As an immediate consequence of Corollary \ref{C:traceclass}, $\Tr(e^{-t\AAA_1^2}-e^{-t\AAA_0^2})$ as well as $\Tr(\AAA_1e^{-t\AAA_1^2}- \AAA_0e^{-t\AAA_0^2})$ are independent of the decomposition \eqref{E:space decomp}.
\end{remark}

\subsection{The short time asymptotic expansions of the relative heat traces}\label{SS:STexpan}

We now study the integrands on the right side of \eqref{E:traceclass-1} and \eqref{E:traceclass-2} using heat kernel estimates. We first construct parametrices for the heat kernels $\KK_0$ and $\KK_1$, then use them to give short time asymptotic expansion for the relative heat trace.

We first introduce a set of cut-off functions (see Fig. 1). Let $\phi,\psi$ be smooth functions on $Y_0$ such that they are both equal to 1 outside a compact set containing $K_0$ and equal to 0 on a smaller compact set containing $K_0$. Moreover, $\supp(1-\psi)\cap\supp\phi=\emptyset$. Now that $\phi=\psi=0$ on $K_0$, we can replace $K_0$ by $K_1$ and also view them as functions on $Y_1$. We still use $\phi$ and $\psi$ to denote them. For $j=0,1$, let $\gamma_j$ be a compactly supported function on $Y_j$ with $\supp(1-\gamma_j)\cap\supp(1-\phi)=\emptyset$.

\begin{center}
\begin{overpic}[height=3.6cm]{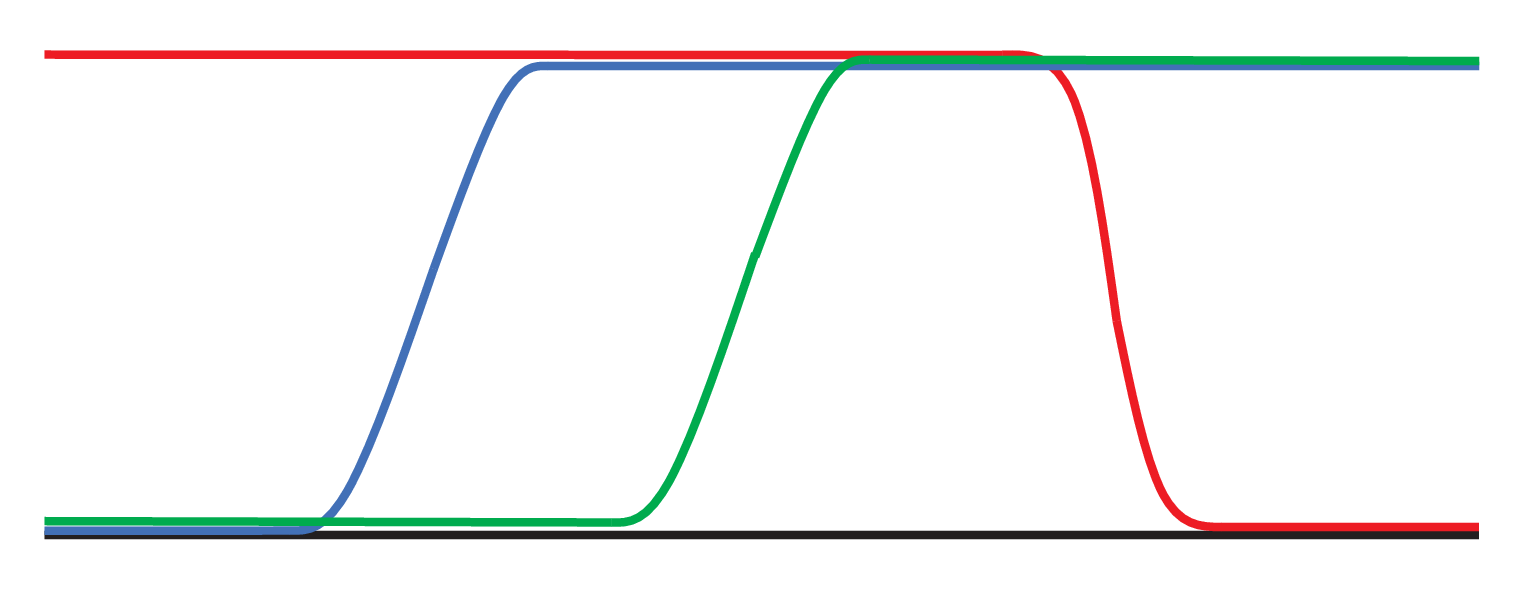}
\put(8,0){$K_j$}
\put(52,0){$U$}
\put(30,20){$\psi$}
\put(50,20){$\phi$}
\put(74,20){$\gamma_j$}
\end{overpic}\\
Figure 1. The cut-off functions.
\end{center}

Let $\tilde{K}_j$ be a compact subset of $Y_j$ containing a neighborhood of $\supp\gamma_j$ with smooth boundary. One can find a closed manifold $\tilde{Y}_j$ which contains $\tilde{K}_j$ and a Dirac-type operator $\tilde{\AAA}_j$ on $\tilde{Y}_j$ such that the restriction of $\tilde{\AAA}_j$ to $\supp\gamma_j$ is $\AAA_j$.\footnote{Basically, one can deform the Riemannian metric, Dirac bundle structure, etc in a small collar neighborhood of $\p\tilde{K}_j$ to make the operator $\AAA_j$ a product form near $\p\tilde{K}_j$, see \cite[Chapter 9]{BoosWoj93book}, \cites{BrCecchini18,BrMaschler19}, and then take the closed double of it.} $\gamma_j$ can also be viewed as a function on $\tilde{Y}_j$ with $\gamma_j=0$ outside $\tilde{K}_j$. Let $\tilKK_j$ be the kernel of the heat operator $e^{-t\tilde{\AAA}_j^2}$. Then one can get a parametrix for the kernel of $e^{-t\AAA_j^2}$ as
\begin{equation}\label{E:parametrix}
\EE_j(y,z;t)\;:=\;\gamma_j(y)\,\tilKK_j(y,z;t)\,(1-\phi(z))\,+\,\psi(y)\,\KK_0(y,z;t)\,\phi(z).
\end{equation}

In what follows, we will fix the choice of the cut-off functions $\phi,\psi,\gamma_j$ and the compact subset $\tilde{K}_j\Subset Y_j$.

\begin{lemma}\label{L:Duhamel}
For $j=0,1$ and $t>0$,
\begin{align*}
&\KK_j(y,z;t)-\EE_j(y,z;t) \\
&\hspace{1.6cm}=\;-\int_0^t\int_{Y_j}\KK_j(y,w;s)\Big(\frac{\p}{\p t}+\AAA_j^2\Big)\EE_j(w,z;t-s)\,dwds, \\
&\AAA_j\KK_j(y,z;t)-\AAA_j\EE_j(y,z;t) \\
&\hspace{1.6cm}=\;-\int_0^t\int_{Y_j}\AAA_j\KK_j(y,w;s)\Big(\frac{\p}{\p t}+\AAA_j^2\Big)\EE_j(w,z;t-s)\,dwds.
\end{align*}
\end{lemma}

\begin{proof}
We only show the first equation. The second one follows from the first. From the construction of $\EE_j$, one has that $\EE_j(y,z;t)$ is the kernel of the identity map when $t\to0$. By Duhamel's principle (cf. \cite[Lemma 22.11]{BoosWoj93book}),
\begin{align*}
&\quad\KK_j(y,z;t)-\EE_j(y,z;t) \\
=&\;\int_{Y_j}\big(\KK_j(y,w;t)\EE_j(w,z;0)-\KK_j(y,w;0)\EE_j(w,z;t)\big)\,dw \\
=&\;\int_{Y_j}\int_0^t\frac{\p}{\p s}\big(\KK_j(y,w;s)\EE_j(w,z;t-s)\big)\,dsdw \\
=&\;\int_0^t\int_{Y_j}\Big[\big(-\AAA_j^2\KK_j(y,w;s)\big)\EE_j(w,z;t-s)-\KK_j(y,w;s)\frac{\p}{\p t}\EE_j(w,z;t-s)\Big]\,dwds \\
=&\;-\int_0^t\int_{Y_j}\KK_j(y,w;s)\Big(\frac{\p}{\p t}+\AAA_j^2\Big)\EE_j(w,z;t-s)\,dwds.
\end{align*}
This completes the proof.
\end{proof}

\begin{lemma}\label{L:STest}
For $j=0,1$, assume the cut-off functions $\phi,\psi,\gamma_j$ and the compact subset $\tilde{K}_j\Subset Y_j$ are fixed. Then there exist constants $\alpha_j,\beta_j>0$ such that as $t\to0$,
\begin{align}
\int_{Y_j}\big|\KK_j(y,y;t)-\EE_j(y,y;t)\big|\,dy&\;\le\;\alpha_j e^{-\beta_j/t}, \label{E:STest-1}\\
\int_{Y_j}\big|\AAA_j\KK_j(y,y;t)-\AAA_j\EE_j(y,y;t)\big|\,dy&\;\le\;\alpha_j e^{-\beta_j/t}.\label{E:STest-2}
\end{align}
\end{lemma}

\begin{proof}
We show \eqref{E:STest-2}. \eqref{E:STest-1} is similar. Note that the integrand has been computed in Lemma \ref{L:Duhamel}. By the definition of $\EE_j$ \eqref{E:parametrix}, one can write
\[
\AAA_j\KK_j(y,w;s)\Big(\frac{\p}{\p t}+\AAA_j^2\Big)\EE_j(w,z;t-s)
\]
as the sum of
\begin{equation}\label{E:kernel-parametrix-1}
\AAA_j\KK_j(y,w;s)\Big(\frac{\p}{\p t}+\AAA_j^2\Big)\big[\gamma_j(w)\,\tilKK_j(w,z;t-s)\,(1-\phi(z))\big]
\end{equation}
and
\begin{equation}\label{E:kernel-parametrix-2}
\AAA_j\KK_j(y,w;s)\Big(\frac{\p}{\p t}+\AAA_j^2\Big)\big[\psi(w)\,\KK_0(w,z;t-s)\,\phi(z)\big].
\end{equation}
Direct computation shows that
\begin{align*}
\eqref{E:kernel-parametrix-1}\;=\;&\AAA_j\KK_j(y,w;s)\Big[\AAA_j\big(\rmc(d\gamma_j)\tilKK_j(w,z;t-s)\big)+\rmc(d\gamma_j)\AAA_j\tilKK_j(w,z;t-s)\Big](1-\phi(z)) \\
=\;&\AAA_j^2\KK_j(y,w;s)\rmc(d\gamma_j)\tilKK_j(w,z;t-s)(1-\phi(z)) \\
&+\,\AAA_j\KK_j(y,w;s)\rmc(d\gamma_j)\AAA_j\tilKK_j(w,z;t-s)(1-\phi(z)),
\end{align*}
where $\rmc(\cdot)$ denotes the Clifford multiplication.
Similarly,
\begin{align*}
\eqref{E:kernel-parametrix-2}\;=\;&\AAA_j^2\KK_j(y,w;s)\rmc(d\psi)\KK_0(w,z;t-s)\phi(z) \\
&+\,\AAA_j\KK_j(y,w;s)\rmc(d\psi)\AAA_j\KK_0(w,z;t-s)\phi(z).
\end{align*}
Now we use Lemma \ref{L:Duhamel} to reduce the desired estimate to the following four parts.
\begin{align*}
I_1&\,:=\,\int_{Y_j}\int_0^t\int_{Y_j}\big|\AAA_j^2\KK_j(y,w;s)\rmc(d\gamma_j)\tilKK_j(w,y;t-s)(1-\phi(y))\big|\,dwdsdy, \\
I_2&\,:=\,\int_{Y_j}\int_0^t\int_{Y_j}\big|\AAA_j\KK_j(y,w;s)\rmc(d\gamma_j)\AAA_j\tilKK_j(w,y;t-s)(1-\phi(y))\big|\,dwdsdy, \\
I_3&\,:=\,\int_{Y_j}\int_0^t\int_{Y_j}\big|\AAA_j^2\KK_j(y,w;s)\rmc(d\psi)\KK_0(w,y;t-s)\phi(y)\big|\,dwdsdy, \\
I_4&\,:=\,\int_{Y_j}\int_0^t\int_{Y_j}\big|\AAA_j\KK_j(y,w;s)\rmc(d\psi)\AAA_j\KK_0(w,y;t-s)\phi(y)\big|\,dwdsdy.
\end{align*}

We show the estimate for $I_3$. For convenience, we set $t\le1$. Let $K_{d\psi}$ (with respect to $w$) and $K_\phi$ (with respect to $y$) be the supports of $d\psi$ and $\phi$, respectively. Then
\begin{align*}
I_3&\,\le\,C_1\int_0^t\int_{K_{d\psi}}\int_{K_\phi}\big|\AAA_j^2\KK_j(y,w;s)\big|\cdot\big|\KK_0(w,y;t-s)\big|\,dydwds \\
&\,\le\,C_1\int_0^t\int_{K_{d\psi}}\bigg(\int_{K_\phi}\big|\AAA_j^2\KK_j(y,w;s)\big|^2\,dy\bigg)^{1/2} \\
&\hspace{3cm}\cdot\bigg(\int_{K_\phi}\big|\KK_0(w,y;t-s)\big|^2\,dy\bigg)^{1/2}\,dwds \\
&\,\le\,C_1\int_0^t\bigg(\int_{K_{d\psi}}\int_{K_\phi}\big|\AAA_j^2\KK_j(y,w;s)\big|^2\,dydw\bigg)^{1/2} \\
&\hspace{2.3cm}\cdot\bigg(\int_{K_{d\psi}}\int_{K_\phi}\big|\KK_0(w,y;t-s)\big|^2\,dydw\bigg)^{1/2}\,ds \\
\end{align*}
Let $d={\rm dist}(K_{d\psi},K_\phi)$. By the properties of our cut-off functions, $d>0$. One can use Fourier transform to write
\begin{align*}
\AAA_j^2\KK_j(y,w;s)\;&=\;\big[\AAA_j^2e^{-s\AAA_j^2}\delta_w\big]\,(y) \\
&=\;\Big[\frac{1}{\sqrt{4\pi s}}\int_{-\infty}^\infty e^{-\xi^2/4s}\AAA_j^2e^{i\xi\AAA_j}\delta_w\,d\xi\Big]\,(y),\qquad y\in K_\phi,\;w\in K_{d\psi},
\end{align*}
where $\delta_w$ is the Dirac delta distribution at $w$. By finite propagation speed \cites{Chernoff73,CheegerGromovTaylor82}, $e^{i\xi\AAA_j}\delta_w$ is supported in a $|\xi|$-neighborhood of $w$ which indicates
\[
\AAA_j^2\KK_j(y,w;s)\;=\;\Big[\frac{1}{\sqrt{4\pi s}}\int_{\RR\setminus(-d,d)}e^{-\xi^2/4s}\AAA_j^2e^{i\xi\AAA_j}\delta_w\,d\xi\Big]\,(y).
\]
Therefore as $t\to0$ (which means $s\to0$),
\begin{align*}
\big|\AAA_j^2\KK_j(y,w;s)\big|\;&\le\;C_2\,e^{-d^2/8s}\,\bigg|\Big[\frac{1}{\sqrt{8\pi s}}\int_{\RR\setminus(-d,d)}e^{-\xi^2/8s}\AAA_j^2e^{i\xi\AAA_j}\delta_w\,d\xi\Big]\,(y)\bigg| \\
&=\;C_2\,e^{-d^2/16s}\,\Big|\big[e^{-d^2/16s}\AAA_j^2e^{-s\AAA_j^2/2}\delta_w\big]\,(y)\Big|.
\end{align*}
For a fixed integer $p>n/2$, by the local Sobolev embedding theorem, $\delta_w\in H_{\AAA_j}^{-p}(Y_j,E_j)$. Note that the smoothing operator $e^{-d^2/16s}\AAA_j^2e^{-s\AAA_j^2/2}$ is uniformly bounded from $H_{\AAA_j}^{-p}(Y_j,E_j)$ to $L^2(Y_j,E_j)$ for $0<s\le1$. Hence
\begin{align*}
\int_{K_\phi}\big|\AAA_j^2\KK_j(y,w;s)\big|^2\,dy\;&\le\;C_3\,e^{-d^2/8s}\,\Big\Vert\big[e^{-d^2/16s}\AAA_j^2e^{-s\AAA_j^2/2}\delta_w\big]\,(y)\Big\Vert_{L^2}^2 \\
&\le\;C_4\,e^{-d^2/8s}\,\Vert\delta_w(y)\Vert_{H_{\AAA_j}^{-p}}^2
\;\le\;C_5(w)\,e^{-d^2/8s},
\end{align*}
where $C_5(w)$ depends only on a local Sobolev embedding constant in $w$, thus is bounded on the compact set $K_{d\psi}$. So we obtain
\[
\int_{K_{d\psi}}\int_{K_\phi}\big|\AAA_j^2\KK_j(y,w;s)\big|^2\,dydw\;\le\;C_6\,e^{-d^2/8s}.
\]
By a similar reason,
\[
\int_{K_{d\psi}}\int_{K_\phi}\big|\KK_0(w,y;t-s)\big|^2\,dydw\;\le\;C_7\,e^{-d^2/8(t-s)}.
\]
Hence
\[
I_3\;\le\;C_8\int_0^te^{-d^2t/8s(t-s)}\,ds\;\le\;C_8\int_0^te^{-d^2/8t}\,ds\;\le\;C_9\,e^{-C_{10}/t},\qquad\mbox{as }t\to 0.
\]

The estimates for the other three terms are totally analogous. Combining them immediately yields \eqref{E:STest-2}.
\end{proof}

Lemma \ref{L:STest} tells us that the difference between the heat kernel and the parametrix is negligible when considering the (global) short time asymptotic expansions. Recall that we have defined $\tilKK_j$ to be the kernel of $e^{-t\tilde{\AAA}_j^2}$ on the closed $n$-manifold $\tilde{Y}_j$. We use it as an example to remind the short time asymptotic expansion of the heat kernel on a closed manifold.

\begin{proposition}[\cites{Gilkey95book,BransonGilkey92}]\label{P:STexpan-closed}
For $j=0,1$, as $t\to0$, we have the following asymptotic expansions
\begin{align*}
\Tr(e^{-t\tilA_j^2})\;=\;\int_{\tilde{Y}_j}\tr(\tilKK_j(y,y;t))\,dy&\;\sim\;\sum_{k=0}^\infty a_{j,k}\,t^{(k-n)/2}, \\
\Tr(\tilA_je^{-t\tilA_j^2})\;=\;\int_{\tilde{Y}_j}\tr(\tilA_j\tilKK_j(y,y;t))\,dy&\;\sim\;\sum_{k=0}^\infty b_{j,k}\,t^{(k-n-1)/2}.
\end{align*}
The coefficients $a_{j,k}$ and $b_{j,k}$ are given by
\[
a_{j,k}\;=\;\int_{\tilde{Y}_j}a_{j,k}(y)\,dy,\quad\qquad
b_{j,k}\;=\;\int_{\tilde{Y}_j}b_{j,k}(y)\,dy,
\]
where the densities $a_{j,k}(y)$ and $b_{j,k}(y)$ are local invariants of the jets of the symbol of $\tilA_j$. Moreover, $a_{j,k}(y)=0$ if $k$ is odd; $b_{j,k}(y)=0$ if $k$ is even.
\end{proposition}

\begin{remark}\label{R:STexpan-closed}
The expansion of $\Tr(\tilA_je^{-t\tilA_j^2})$ can be used to analyse the analytic continuation of $\eta(s;\tilA_j)$. In particular, one sees that $\eta(s;\tilA_j)$ has a simple pole at $s=0$ with residue given by $2b_{j,n}/\sqrt{\pi}$. Since we already know that $s=0$ is a regular point, we get $b_{j,n}=0$. However, generically $b_{j,n}(y)\ne0$ (cf. \cite{BransonGilkey92}).
\end{remark}

We now use $\EE_j$ to replace $\KK_j$ to find the asymptotic expansion of $\Tr(e^{-t\AAA_1^2}-e^{-t\AAA_0^2})=\int_{Y_0\cup_U Y_1}\tr(\KK_1(y,y;t)-\KK_0(y,y;t))dy$. By the construction \eqref{E:parametrix} of $\EE_j$,
\[
\EE_1(y,y;t)\,-\,\EE_0(y,y;t)\;=\;(\tilKK_1(y,y;t)-\tilKK_0(y,y;t))\,(1-\phi(y)).
\]
In other words, it is entirely determined by the expansions of $\tilKK_j(y,y;t)$. This is also true for the asymptotic expansion of $\Tr(\AAA_1e^{-t\AAA_1^2}-\AAA_0e^{-t\AAA_0^2})$. It then follows from Proposition \ref{P:STexpan-closed} that

\begin{proposition}\label{P:STexpan-relative}
Consider $\AAA_j$ which satisfies Assumption \ref{A:traceclass} to be an operator on $Y_0\cup_U Y_1$. As $t\to0$, we have the following asymptotic expansions
\begin{align*}
\Tr(e^{-t\AAA_1^2}-e^{-t\AAA_0^2})&\;\sim\;\sum_{k=0}^\infty a_k\,t^{(k-n)/2}, \\
\Tr(\AAA_1e^{-t\AAA_1^2}-\AAA_0e^{-t\AAA_0^2})&\;\sim\;\sum_{k=0}^\infty b_k\,t^{(k-n-1)/2},
\end{align*}
where
\begin{align*}
a_k&\;=\;\int_{Y_0\cup_U Y_1}(a_{1,k}(y)-a_{0,k}(y))\,(1-\phi(y))\,dy \\
&\;=\;\int_{\tilde{K}_1}a_{1,k}(y)\,(1-\phi(y))\,dy\,-\,\int_{\tilde{K}_0}a_{0,k}(y)\,(1-\phi(y))\,dy, \\
b_k&\;=\;\int_{Y_0\cup_U Y_1}(b_{1,k}(y)-b_{0,k}(y))\,(1-\phi(y))\,dy \\
&\;=\;\int_{\tilde{K}_1}b_{1,k}(y)\,(1-\phi(y))\,dy\,-\,\int_{\tilde{K}_0}b_{0,k}(y)\,(1-\phi(y))\,dy, \\
\end{align*}
and $a_{j,k}(y)$ and $b_{j,k}(y)$ are as in Proposition \ref{P:STexpan-closed}.
In particular, $a_k=0$ if $k$ is odd; $b_k=0$ if $k$ is even. Also $b_n=0$.
\end{proposition}

\begin{proof}
We only need to show $b_n=0$, the rest is clear. By the selection of $\tilde{K}_0$ and $\tilde{K}_1$, one  can enclose $\tilde{K}_0$ and $\tilde{K}_1$ by a common compact set $K'$ and get two closed smooth manifolds
\[
L_0\;=\;\tilde{K}_0\cup K',\qquad L_1\;=\;\tilde{K}_1\cup K'.
\]
The operators $\AAA_j$ ($j=0,1$) can then be extended to Dirac-type operators $\hat{A}_j$ on $L_j$. As in the beginning of this subsection, we choose smooth cut-off functions $\hat{\phi},\hat{\psi}$ and $\hat{\gamma}_j$ ($j=0,1$) on $L_0$ and $L_1$ such that they are equal to $\phi,\psi$ and $\gamma_j$, respectively on $\tilde{K}_j$. Similar to \eqref{E:parametrix} we construct a parametrix $\hat{\EE}_j$ for the kernel $\hat{\KK}_j$ of $e^{-t\hat{\AAA}_j^2}$ such that $\hat{\EE}_j=\tilde{\KK}_j$ on $\tilde{K}_j$, where $\tilde{\KK}_j$ is as in \eqref{E:parametrix}. Then Lemma \ref{L:STest} holds obviously with $\KK_j$ and $\EE_j$ replaced by $\hat{\KK}_j$ and $\hat{\EE}_j$, respectively.

From the above discussion, it can be deduced that $\Tr(e^{-t\AAA_1^2}-e^{-t\AAA_0^2})$ and $\Tr(e^{-t\hat{\AAA}_1^2}-e^{-t\hat{\AAA}_0^2})$ (resp. $\Tr(\AAA_1e^{-t\AAA_1^2}-\AAA_0e^{-t\AAA_0^2})$ and $\Tr(\hat{\AAA}_1e^{-t\hat{\AAA}_1^2}-\hat{\AAA}_0e^{-t\hat{\AAA}_0^2})$) have the same asymptotic expansions as $t\to0$. Note that $\hat{\AAA}_j$ is defined on a closed manifold. Let
\[
\Tr(\hat{\AAA}_je^{-t\hat{\AAA}_j^2})\;\sim\;\sum_{k=0}^\infty\hat{b}_{j,k}\,t^{(k-n-1)/2},
\]
be the asymptotic expansion as $t\to0$. By Remark \ref{R:STexpan-closed}, $\hat{b}_{j,n}=0$. Therefore $b_n=\hat{b}_{1,n}-\hat{b}_{0,n}=0$.
\end{proof}

\subsection{The large time asymptotic behavior of the relative heat traces}\label{SS:LTbehav}

In this subsection we explore the behavior of $\Tr(e^{-t\AAA_1^2}-e^{-t\AAA_0^2})$ and $\Tr(\AAA_1e^{-t\AAA_1^2}-\AAA_0e^{-t\AAA_0^2})$ as $t\to\infty$ by the method of spectral shift function.

We first recall the basic facts about spectral shift function in our setting. (cf. \cite[Section 2]{Muller98}.)

\begin{lemma}\label{L:sp shift}
Let $\AAA_0$ and $\AAA_1$ be two self-adjoint Dirac-type operators which coincide and are invertible at infinity, and satisfy Assumption \ref{A:traceclass}. Suppose $(a,b)\subset\RR$ is an interval which does not intersect the spectra of $\AAA_0$ and $\AAA_1$. Then there exists a unique real valued locally integrable function $\sigma(\lambda)=\sigma(\lambda;\AAA_0,\AAA_1)$ on $\RR$, called the \emph{spectral shift function} of $\AAA_0,\AAA_1$, such that
\begin{enumerate}
\item $\sigma(\lambda)=0$ for $\lambda\in(a,b)$;
\item for all $\varphi\in C_0^\infty(\RR)$, $\varphi(\AAA_1)-\varphi(\AAA_0)$ is a trace class operator and
\begin{equation}\label{E:sp shift}
\Tr\big(\varphi(\AAA_1)-\varphi(\AAA_0)\big)\;=\;\int_\RR\varphi'(\lambda)\,\sigma(\lambda)\,d\lambda;
\end{equation}
\item
\begin{align}
\Tr\big(e^{-t\AAA_1^2}-e^{-t\AAA_0^2}\big)\; & =\;\int_\RR\frac{d}{d\lambda}(e^{-t\lambda^2})\,\sigma(\lambda)\,d\lambda, \label{E:sp shift-heat-1}\\
\Tr\big(\AAA_1e^{-t\AAA_1^2}-\AAA_0e^{-t\AAA_0^2}\big)\; & =\;\int_\RR\frac{d}{d\lambda}(\lambda e^{-t\lambda^2})\,\sigma(\lambda)\,d\lambda. \label{E:sp shift-heat-2}
\end{align}
\end{enumerate}
\end{lemma}

Since $\AAA_j$ ($j=0,1$) has discrete spectrum near 0, one can find a $\delta>0$ which is a lower bound for the absolute values of non-zero spectra of $\AAA_0$ and $\AAA_1$. Therefore the interval in Lemma \ref{L:sp shift} can be chosen to be $(-\delta,0)$. Fix $t>0$, let $f_t\in C_0^\infty(\RR)$ such that $f_t(x)=e^{-tx^2}$ for $|x|\le\delta/2$ and $f_t(x)=0$ for $|x|\ge\delta$. Then by \eqref{E:sp shift},
\[
\Tr\big(f_t(\AAA_1)-f_t(\AAA_0)\big)\;=\;\int_0^\delta f_t'(\lambda)\,\sigma(\lambda)\,d\lambda.
\]
Note that $f_t(\AAA_j)$ is equal to $e^{-t\AAA_j^2}$ on the finite-dimensional subspace of $\HH$ in \eqref{E:space decomp} spanned by the eigensections of $\AAA_j$ with eigenvalue 0, and equal to 0 elsewhere. Thus it is a trace-class operator on $\HH$ whose trace is $\dim\ker\AAA_j$. It follows that
\[
\int_0^\delta f_t'(\lambda)\,\sigma(\lambda)\,d\lambda\;=\;\dim\ker\AAA_1\,-\,\dim\ker\AAA_0.
\]
Note that $\sigma(\lambda)=\dim\ker\AAA_0-\dim\ker\AAA_1$ satisfies the above equation. By the uniqueness of spectral shift function,
\[
\sigma(\lambda)\;=\;\dim\ker\AAA_0\,-\,\dim\ker\AAA_1,\qquad 0\le\lambda\le\delta.
\]

By \eqref{E:sp shift-heat-1},
\[
\begin{aligned}
\Tr\big(e^{-t\AAA_1^2}-e^{-t\AAA_0^2}\big)\;=\; & \dim\ker\AAA_1\,-\,\dim\ker\AAA_0 \\
& -2t\int_{-\infty}^{-\delta}\lambda e^{-t\lambda^2}\sigma(\lambda)\,d\lambda\,-\,2t\int_\delta^\infty\lambda e^{-t\lambda^2}\tilde{\sigma}(\lambda)\,d\lambda,
\end{aligned}
\]
where $\tilde{\sigma}(\lambda):=\sigma(\lambda)+\dim\ker\AAA_1-\dim\ker\AAA_0$. For $t\ge1$, both integrals on the right hand side can be estimated by
\[
e^{-t\delta^2/2}\int_{-\infty}^{-\delta}|\lambda e^{-t\lambda^2/2}\sigma(\lambda)|\,d\lambda\;\Big(\text{or }e^{-t\delta^2/2}\int_\delta^\infty|\lambda e^{-t\lambda^2/2}\tilde{\sigma}(\lambda)|\,d\lambda\Big)\;\le\;Ce^{-t\delta^2/2}.
\]
Similarly, by \eqref{E:sp shift-heat-2},
\[
\begin{aligned}
&\Tr\big(\AAA_1e^{-t\AAA_1^2}-\AAA_0e^{-t\AAA_0^2}\big) \\
&\qquad=\;
\int_{-\infty}^{-\delta}(1-2t\lambda^2)e^{-t\lambda^2}\sigma(\lambda)\,d\lambda\,+\,\int_\delta^\infty(1-2t\lambda^2)e^{-t\lambda^2}\tilde{\sigma}(\lambda)\,d\lambda.
\end{aligned}
\]
Again both integrals can be estimated by $(C+C't)e^{-t\delta^2/2}$ for $t\ge1$. 

In conclusion, we have shown the following large time exponentially decaying property of the relative heat traces.

\begin{proposition}\label{P:LTbehav}
There exists a constant $C>0$ such that as $t\to\infty$,
\begin{align*}
\Tr\big(e^{-t\AAA_1^2}-e^{-t\AAA_0^2}\big)\;&=\;\dim\ker\AAA_1\,-\,\dim\ker\AAA_0\,+\,O(e^{-Ct}); \\
\Tr\big(\AAA_1e^{-t\AAA_1^2}-\AAA_0e^{-t\AAA_0^2}\big)\;&=\;O(e^{-Ct}).
\end{align*}
\end{proposition}

\section{Relative eta function and relative eta invariant}\label{S:releta}

Let $\AAA_0$ and $\AAA_1$ be two self-adjoint Dirac-type operators which coincide and are invertible at infinity. Assume that they satisfy Assumption \ref{A:traceclass}. In the last section, we have known that $\AAA_1e^{-t\AAA_1^2}-\AAA_0e^{-t\AAA_0^2}$ is a trace-class operator. Then we revealed the asymptotic properties of $\Tr(\AAA_1e^{-t\AAA_1^2}-\AAA_0e^{-t\AAA_0^2})$ as $t\to0$ and $t\to\infty$. With the help of these results, we can now talk about the validity of the definition of the  relative eta function.


Recall that
\begin{equation}\label{E:releta func def}
		\eta(s;\AAA_1,\AAA_0) \;:= \;
	\frac{1}{\Gamma((s+1)/2)}\,\int_0^\infty
	  t^{(s-1)/2}
	     \Tr\big(\AAA_1e^{-t\AAA_1^2}-\AAA_0e^{-t\AAA_0^2}\big) \,dt.
\end{equation}
We split the right hand side into the sum of two terms
\begin{equation}\label{E:releta func1-}
	\frac{1}{\Gamma((s+1)/2)}\,\int_0^1
	  t^{(s-1)/2}
	     \Tr\big(\AAA_1e^{-t\AAA_1^2}-\AAA_0e^{-t\AAA_0^2}\big) \,dt
\end{equation}
and
\begin{equation}\label{E:releta func1+}
	\frac{1}{\Gamma((s+1)/2)}\,\int_1^\infty
	  t^{(s-1)/2}
	     \Tr\big(\AAA_1e^{-t\AAA_1^2}-\AAA_0e^{-t\AAA_0^2}\big) \,dt.
\end{equation}

By Proposition \ref{P:STexpan-relative}, the integral in \eqref{E:releta func1-} is absolutely convergent and holomorphic in the half plane $\Re(s)>n$ and admits a meromorphic continuation to the whole complex plane. While by Proposition \ref{P:LTbehav}, the integral in \eqref{E:releta func1+} is absolutely convergent for $s$ in the whole complex plane. To sum up, $\eta(s;\AAA_1,\AAA_0)$ is well-defined in the half plane $\Re(s)>n$ and admits a meromorphic continuation to the whole complex plane. It has a simple pole at $s=0$ with
\[
\Res_{s=0}\eta(s;\AAA_1,\AAA_0)\;=\;\frac{2}{\sqrt{\pi}}\,b_n\;=\;0,
\]
as in Proposition \ref{P:STexpan-relative}. Therefore we obtain the following regularity result.

\begin{theorem}\label{T:releta func}
Let $\AAA_0$ and $\AAA_1$ be two self-adjoint Dirac-type operators satisfying Assumption \ref{A:traceclass}. Then the relative eta function $\eta(s;\AAA_1,\AAA_0)$ of \eqref{E:releta func def} is regular at $s=0$.
\end{theorem}

This induces the definition of relative eta invariant.

\begin{definition}\label{D:releta-sp}
We call $\eta(0;\AAA_1,\AAA_0)\in\RR$ the \emph{relative eta invariant} associated to $\AAA_0$ and $\AAA_1$.
\end{definition}

\begin{remark}\label{R:releta-general}
In more general cases, for example, on manifolds with boundary which will be discussed later, we may only suppose the asymptotic expansion of $\Tr(\AAA_1e^{-t\AAA_1^2}-\AAA_0e^{-t\AAA_0^2})$ as $t\to0$ has the following general form (cf. \cite[(2.12)]{BruningLesch99})
\begin{equation}\label{E:STexp-general}
\Tr\big(\AAA_1e^{-t\AAA_1^2}-\AAA_0e^{-t\AAA_0^2}\big)\;\sim\;\sum_{\substack{\Re(\alpha)\to\infty \\ 0\le k\le k(\alpha)}}b_{\alpha k}t^\alpha\log^kt,
\end{equation}
where $k\in\ZZ_+$ and $\{\alpha\in\CC:b_{\alpha k}\ne0\text{ for some }0\le k\le k(\alpha)\}$ is a countable subset of $\CC$ whose real parts accumulate at most at $\infty$. Then \eqref{E:releta func1-} is absolutely convergent and holomorphic for $\Re(s)>-2\inf\{\Re(\alpha):b_{\alpha k}\ne0\text{ for some }0\le k\le k(\alpha)\}$ and again admits a meromorphic continuation to the whole complex plane, with poles situated at $s=-2\alpha-1$. In this case the regularity of the relative eta function at zero is not clear but the relative eta invariant can still be defined to be the constant term in the Laurent expansion of $\eta(s;\AAA_1,\AAA_0)$ at $s=0$.
\end{remark}

Compared to Definition \ref{D:releta}, where the invariant $\eta(\AAA_1,\AAA_0)$ is only defined for cobordant self-adjoint strongly Callias-type operators, the relative eta invariant $\eta(0;\AAA_1,\AAA_0)$ defined here works for a broader scope of operators. In the following sections, we will work on finding the relationship between these two invariants when they are both defined. To avoid confusions, we will often call $\eta(0;\AAA_1,\AAA_0)$ the \emph{spectral} relative eta invariant and $\eta(\AAA_1,\AAA_0)$ the \emph{index-theoretic} relative eta invariant. Like \eqref{E:reduced releta-1}, one can also define the reduced (spectral) relative eta invariant
\begin{equation}\label{E:reduced releta sp}
\xi(0;\AAA_1,\AAA_0)\;:=\;\frac{1}{2}\,\big(\eta(0;\AAA_1,\AAA_0)\,+\,\dim\ker\AAA_1\,-\,\dim\ker\AAA_0\big).
\end{equation}

\begin{remark}\label{R:reletafuc}
Recall that the eta function of a self-adjoint Dirac-type operator $D$ on a closed manifold can be alternatively defined as \eqref{E:etafunc} for $\Re(s)\gg0$. Unlike it, such an expression of the relative eta function does not exist in general because the individual operator $\AAA_je^{-t\AAA_j^2}$ may not be trace-class.
\end{remark}

One can easily verify the following analogue of Proposition \ref{P:reletainv-prop}.

\begin{proposition}\label{P:reletafunc-prop}
For $\Re(s)>n$ as well as $s=0$, the relative eta function defined above satisfies the following properties
\begin{enumerate}
\item $\eta(s;\AAA_0,\AAA_0)=0$,
\item $\eta(s;\AAA_2,\AAA_1)+\eta(s;\AAA_1,\AAA_0)=\eta(s;\AAA_2,\AAA_0)$.
\end{enumerate}
\end{proposition}

Analogously, one can also talk about the relative zeta function and relative zeta invariant. The \emph{relative zeta function} is defined by
\begin{equation}\label{E:relzeta func def}
\begin{aligned}
		&\zeta(s;\AAA_1,\AAA_0) \\
		&\ := \;
	\frac{1}{\Gamma(s/2)}\,\int_0^\infty
	  t^{s/2-1}
	     \Big(\Tr\big(e^{-t\AAA_1^2}-e^{-t\AAA_0^2}\big)-\dim\ker\AAA_1+\dim\ker\AAA_0\Big)\,dt.
\end{aligned}	     
\end{equation}
Again from Propositions \ref{P:STexpan-relative} and \ref{P:LTbehav}, $\zeta(s;\AAA_1,\AAA_0)$ is holomorphic in the half plane $\Re(s)>n$ and admits a meromorphic continuation to the whole complex plane. In addition, it is regular at $s=0$. Therefore we call $\zeta(0;\AAA_1,\AAA_0)$ the \emph{relative zeta invariant} of $\AAA_0$ and $\AAA_1$.

\section{Variation formulas of the relative eta invariant}\label{S:varFormula}

We analyse the smooth parts of the index-theoretic and spectral relative eta invariants by looking at their variation formulas. Unlike the compact case, things are a little complicated in non-compact setting. We first give a brief review of the results on compact manifolds, which is somehow well-known and can be found, for instance, in \cites{Gilkey95book,Melrose93APS,Muller94}, etc.

\begin{theorem}\label{T:varFormula-cpt}
Let $D_r$ be a smooth one-parameter family of self-adjoint Dirac-type operators on an $n$-dimensional compact manifold $M$. Let
\[
\eta(s;D_r)\;=\;\frac{1}{\Gamma((s+1)/2)}\int_0^\infty
	  t^{(s-1)/2}\Tr\big(D_re^{-tD_r^2}\big)\,dt
\]
be the eta function and $\eta(D_r)=\eta(0;D_r)$ be the eta invariant of $D_r$. Assume as $t\to0$, $\Tr(\dot{D}_re^{-tD_r^2})$ has the following asymptotic expansion
\[
\Tr\big(\dot{D}_re^{-tD_r^2}\big)\;\sim\;\sum_{k=0}^\infty c_k(r)t^{(k-n-1)/2},
\]
where $\dot{D}_r=\frac{d}{dr}D_r$.
\begin{enumerate}
\item Suppose that $\dim\ker D_r$ is constant. Then for $\Re(s)>n$, $\eta(s;D_r)$ depends smoothly on $r$ and
\begin{equation}\label{E:varFormula-cpt-1}
\frac{\p}{\p r}\eta(s;D_r)\;=\;-\frac{s}{\Gamma((s+1)/2)}\int_0^\infty
	  t^{(s-1)/2}\Tr\big(\dot{D}_re^{-tD_r^2}\big)\,dt.
\end{equation}

\item The mod $\ZZ$ reduction $\oeta(D_r)$ of the eta invariant is a smooth function of $r$ and
\begin{equation}\label{E:varFormula-cpt-2}
\frac{d}{dr}\oeta(D_r)\;=\;-\frac{2}{\sqrt{\pi}}c_n(r).
\end{equation}
In particular, when $\dim\ker D_r$ is constant, $\oeta(D_r)$ can be replaced by $\eta(D_r)$ in \eqref{E:varFormula-cpt-2}.
\end{enumerate}
\end{theorem}

\begin{remark}\label{R:varFormula-cpt}
(1) Here the manifold can be with or without boundary. If $M$ has a boundary, then one can impose an APS type boundary condition and require that the operator does not depend on $r$ near the boundary (see \cite[Section 2]{Muller94}).

(2) In the theorem, notice that $c_k(r)$ can be computed locally as $c_k(r)=\int_Mc_k(r)(x)dx$, where $c_k(r)(x)$ represents the coefficient in the asymptotic expansion of the pointwise trace $\tr(\dot{D}_re^{-tD_r^2}(x,x))$ as $t\to0$. Therefore one sees that even though the eta invariant is not locally computable, its variation is actually local. Moreover, one can use the fact that $\frac{\p}{\p r}\eta(s;D_r)$ is holomorphic at $s=0$ to deduce that the residue of $\eta(s;D_r)$ at $s=0$ does not depend on $r$, hence is a homotopy invariant.
\end{remark}

\subsection{Variation formula of the index-theoretic relative eta invariant}\label{SS:varFormula-ind}

Now we establish a formula analogous to \eqref{E:varFormula-cpt-2} for the index-theoretically defined relative eta invariant $\eta(\AAA_1,\AAA_0)$. Let $\AAA_0$ be a self-adjoint strongly Callias-type operator on a complete $n$-dimensional manifold $(Y_0,E_0)$. Let $\AAA_{1,r}$ be a smooth family of self-adjoint strongly Callias-type operators on $(Y_1,E_1)$ which is cobordant to $\AAA_0$. Then $\dot{\AAA}_{1,r}$ vanishes outside a compact set. Suppose we have the following asymptotic expansion for the pointwise trace
\[
\tr(\dot{\AAA}_{1,r}e^{-t\AAA_{1,r}^2}(y,y))\;\sim\;\sum_{k=0}^\infty c_{1,k}(r)(y)t^{(k-n-1)/2},\qquad\text{as }t\to0.
\]
It follows that $c_{1,k}(r)(y)$ is compactly supported over $Y_1$.

\begin{theorem}\label{T:varFormula-noncpt-ind}
Under the above hypothesis, $\oeta(\AAA_{1,r},\AAA_0)$ depends smoothly on $r$ and
\begin{equation}\label{E:varFormula-noncpt-ind}
\frac{d}{dr}\oeta(\AAA_{1,r},\AAA_0)\;=\;-\frac{2}{\sqrt{\pi}}\int_{Y_1}c_{1,n}(r)(y)\,dy.
\end{equation}
\end{theorem}

\begin{remark}\label{R:varFormula-noncpt-ind}
As one shall see in Lemma \ref{L:varTraceclass} below, if $\AAA_0$ and $\AAA_{1,r}$ satisfy Assumption \ref{A:traceclass}, then $\dot{\AAA}_{1,r}e^{-t\AAA_{1,r}^2}$ is in fact a trace-class operator. Thus the integral on the right hand side of \eqref{E:varFormula-noncpt-ind} is globally the coefficient corresponding to $t^{-1/2}$ in the short time asymptotic expansion of $\Tr(\dot{\AAA}_{1,r}e^{-t\AAA_{1,r}^2})$.
\end{remark}

\begin{proof}
By assumption, we can find $(X,\DD_r)$, an almost compact cobordism between $\AAA_0$ and $\AAA_{1,r}$ (cf. Definition \ref{D:almost compact cobordism}), where $X$ is independent of $r$ with $\p X=Y_0\sqcup Y_1$ and $\DD_r=\DD_r^+\oplus\DD_r^-$ is a $\ZZ_2$-graded strongly Callias-type operator that depends smoothly on $r$. The restriction of $\DD_r^+=D_r^++\Psi_r^+$ (resp. $\DD_r^-=D_r^-+\Psi_r^-$) to $Y_0$ is $\AAA_0$ (resp. $\AAA_0^\sharp=\rmc(\nu_{Y_0})\AAA_0\rmc(\nu_{Y_0})$) and to $Y_1$ is $-\AAA_{1,r}$ (resp. $-\AAA_{1,r}^\sharp=-\rmc(\nu_{Y_1})\AAA_{1,r}\rmc(\nu_{Y_1})$).
Then by Definition \ref{D:releta},
\[
\frac{d}{dr}\oeta(\AAA_{1,r},\AAA_0)\;=\;-2\int_X\frac{\p}{\p r}\AS(\DD_r^+)
\]
and it suffices to prove
\[
\int_X\frac{\p}{\p r}\AS(\DD_r^+)\;=\;\frac{1}{\sqrt{\pi}}\int_{Y_1}c_{1,n}(r)(y)\,dy.
\]

Recall that $\AS(\DD_r^+)(x)$ is equal to the constant term in the short time asymptotic expansion of $\tr(e^{-t\DD_r^-\DD_r^+}(x,x))-\tr(e^{-t\DD_r^+\DD_r^-}(x,x))$, where $e^{-t\DD_r^\mp\DD_r^\pm}(\cdot,\cdot)$ denotes the heat kernel on the double of $X$. We have
\[
\begin{aligned}
& \frac{\p}{\p r}\tr(e^{-t\DD_r^-\DD_r^+}(x,x))\,=\,-\tr\big[t(\dot{\DD}_r^-\DD_r^++\DD_r^-\dot{\DD}_r^+)\,e^{-t\DD_r^-\DD_r^+}(x,x)\big] \\
=\; & -\tr\big[t(\dot{\DD}_r^-e^{-t\DD_r^+\DD_r^-}\DD_r^+)(x,x)\big]-\tr\big[t\DD_r^-\dot{\DD}_r^+e^{-t\DD_r^-\DD_r^+}(x,x)\big] \\
=\; & -\tr\big[t(\dot{\DD}_r^-e^{-t\DD_r^+\DD_r^-}D_r^+)(x,x)\big]-\tr\big[t(\dot{\DD}_r^-e^{-t\DD_r^+\DD_r^-}\Psi_r^+)(x,x)\big] \\
& -\tr\big[t(D_r^-\dot{\DD}_r^+e^{-t\DD_r^-\DD_r^+})(x,x)\big]-\tr\big[t(\Psi_r^-\dot{\DD}_r^+e^{-t\DD_r^-\DD_r^+})(x,x)\big],
\end{aligned}
\]
where we use the operator equality $\DD_r^+e^{-t\DD_r^-\DD_r^+}=e^{-t\DD_r^+\DD_r^-}\DD_r^+$ to get the second line. Likewise,
\[
\begin{aligned}
& \frac{\p}{\p r}\tr(e^{-t\DD_r^+\DD_r^-}(x,x)) \\
=\; & -\tr\big[t(\dot{\DD}_r^+e^{-t\DD_r^-\DD_r^+}D_r^-)(x,x)\big]-\tr\big[t(\dot{\DD}_r^+e^{-t\DD_r^-\DD_r^+}\Psi_r^-)(x,x)\big] \\
& -\tr\big[t(D_r^+\dot{\DD}_r^-e^{-t\DD_r^+\DD_r^-})(x,x)\big]-\tr\big[t(\Psi_r^+\dot{\DD}_r^-e^{-t\DD_r^+\DD_r^-})(x,x)\big].
\end{aligned}
\]
Note that
\begin{multline*}
\tr\big[t(\dot{\DD}_r^-e^{-t\DD_r^+\DD_r^-}\Psi_r^+)(x,x)\big]\;=\;\tr\big[t(\dot{\DD}_r^-e^{-t\DD_r^+\DD_r^-})(x,x)\Psi_r^+(x)\big] \\
\;=\;\tr\big[t\,\Psi_r^+(x)(\dot{\DD}_r^-e^{-t\DD_r^+\DD_r^-})(x,x)\big]\;=\;\tr\big[t(\Psi_r^+\dot{\DD}_r^-e^{-t\DD_r^+\DD_r^-})(x,x)\big].
\end{multline*}
Similarly
\[
\tr\big[t(\dot{\DD}_r^+e^{-t\DD_r^-\DD_r^+}\Psi_r^-)(x,x)\big]\;=\;\tr\big[t(\Psi_r^-\dot{\DD}_r^+e^{-t\DD_r^-\DD_r^+})(x,x)\big].
\]
With these cancellations, we obtain
\begin{align}
& \frac{\p}{\p r}\big[\tr(e^{-t\DD_r^-\DD_r^+}(x,x))-\tr(e^{-t\DD_r^+\DD_r^-}(x,x))\big] \nonumber \\
=\; & -\tr\big[t(\dot{\DD}_r^-e^{-t\DD_r^+\DD_r^-}D_r^+)(x,x)\big]+\tr\big[t(\dot{\DD}_r^+e^{-t\DD_r^-\DD_r^+}D_r^-)(x,x)\big] \label{E:varRelativetrace-1} \\
& -\tr\big[t(D_r^-\dot{\DD}_r^+e^{-t\DD_r^-\DD_r^+})(x,x)\big]+\tr\big[t(D_r^+\dot{\DD}_r^-e^{-t\DD_r^+\DD_r^-})(x,x)\big]. \label{E:varRelativetrace-2}
\end{align}
We now take a closer look at \eqref{E:varRelativetrace-1}. By the relation between the kernel of adjoint operators,
\[
\begin{aligned}
\tr\big[(\dot{\DD}_r^-e^{-t\DD_r^+\DD_r^-}D_r^+)(x,x)\big]\; & =\;\tr\big[(\dot{\DD}_r^-e^{-t\DD_r^+\DD_r^-}D_r^+)(x,x')|_{x=x'}\big] \\
& =\;\tr\big[\overline{(D_{r,x'}^-e^{-t\DD_r^+\DD_r^-}\dot{\DD}_r^+)(x',x)}^*|_{x=x'}\big] \\
& =\;-\tr\big[D_{r,x'}^+\overline{(e^{-t\DD_r^+\DD_r^-}\dot{\DD}_r^+)(x',x)}^*|_{x=x'}\big] \\
& =\;-\tr\big[D_{r,x'}^+(\dot{\DD}_r^-e^{-t\DD_r^+\DD_r^-})(x,x')|_{x=x'}\big],
\end{aligned}
\]
where $D_{r,x'}$ denotes action on the $x'$ variable and $\overline{\cdot}^*$ means taking adjoint. Similarly,
\[
\tr\big[(\dot{\DD}_r^+e^{-t\DD_r^-\DD_r^+}D_r^-)(x,x)\big]\;=\;-\tr\big[D_{r,x'}^-(\dot{\DD}_r^+e^{-t\DD_r^-\DD_r^+})(x,x')|_{x=x'}\big].
\]
On the other hand, the components in \eqref{E:varRelativetrace-2} are
\[
\begin{aligned}
(D_r^-\dot{\DD}_r^+e^{-t\DD_r^-\DD_r^+})(x,x)\; &= \;D_{r,x}^-(\dot{\DD}_r^+e^{-t\DD_r^-\DD_r^+})(x,x')|_{x=x'}, \\
(D_r^+\dot{\DD}_r^-e^{-t\DD_r^+\DD_r^-})(x,x)\; &= \;D_{r,x}^+(\dot{\DD}_r^-e^{-t\DD_r^+\DD_r^-})(x,x')|_{x=x'}.
\end{aligned}
\]
We write the Dirac operator $D_r^\pm$ in the form $\sum_{j=1}^{n+1}\rmc_r^\pm(e_j)\nabla_{e_j}^E(r)$ with respect to an orthonormal moving frame $\{e_j\}_{j=1}^{n+1}$. Combining \eqref{E:varRelativetrace-1} and \eqref{E:varRelativetrace-2} and applying the arguments of \cite[Lemma 7.6]{Mueller96}, we obtain a local McKean--Singer type formula
\[
\begin{aligned}
& \frac{\p}{\p r}\big[\tr(e^{-t\DD_r^-\DD_r^+}(x,x))-\tr(e^{-t\DD_r^+\DD_r^-}(x,x))\big] \nonumber \\
=\; & -\tr\big[t(D_{r,x}^-+D_{r,x'}^-)(\dot{\DD}_r^+e^{-t\DD_r^-\DD_r^+})(x,x')|_{x=x'}\big] \\
& +\tr\big[t(D_{r,x}^++D_{r,x'}^+)(\dot{\DD}_r^-e^{-t\DD_r^+\DD_r^-})(x,x')|_{x=x'}\big] \\
=\; & -t\operatorname{div}V_r,
\end{aligned}
\]
where
\[
V_r\;=\;\sum_{j=1}^{n+1}\Big[\tr\big[\rmc_r^-(e_j)\dot{\DD}_r^+e^{-t\DD_r^-\DD_r^+}(x,x)\big]-\tr\big[\rmc_r^+(e_j)\dot{\DD}_r^-e^{-t\DD_r^+\DD_r^-}(x,x)\big]\Big]e_j.
\]

From the discussion of $(X,\DD_r)$ at the beginning of the proof, it follows that $\dot{\DD}_r^\pm=0$ near $Y_0$ while $\rmc^-(e_{n+1})\dot{\DD}_r^+=-\dot{\AAA}_{1,r}$ and $\rmc^+(e_{n+1})\dot{\DD}_r^-=-\dot{\AAA}_{1,r}^\sharp$ when restricting to $Y_1$, where $e_{n+1}$ is the outward unit normal vector near the boundary. Here we suppress the subscript ``$r$'' in $\rmc^\pm$ because the Clifford multiplication by $e_{n+1}$ is independent of $r$. Furthermore, for $x=(u,y)$ which is close to $Y_1$, as $t\to0$
\begin{align*}
e^{-t\DD_r^-\DD_r^+}(x,x)\,&\sim\,\frac{1}{2\sqrt{\pi t}}e^{-t\AAA_{1,r}^2}(y,y); \\
e^{-t\DD_r^+\DD_r^-}(x,x)\,&\sim\,\frac{1}{2\sqrt{\pi t}}e^{-t(\AAA_{1,r}^\sharp)^2}(y,y) \\
&=\,-\frac{1}{2\sqrt{\pi t}}(\rmc^+(e_{n+1})e^{-t\AAA_{1,r}^2}\rmc^-(e_{n+1}))(y,y).
\end{align*}
Integrating over $X$, we get
\begin{equation}\label{E:varFormula-expansion}
\int_X\frac{\p}{\p r}\big[\tr(e^{-t\DD_r^-\DD_r^+}(x,x))-\tr(e^{-t\DD_r^+\DD_r^-}(x,x))\big]\,\sim\,t\int_{Y_1}\frac{1}{\sqrt{\pi t}}\tr(\dot{\AAA}_{1,r}e^{-t\AAA_r^2}(y,y))
\end{equation}
as $t\to0$.
Let
\[
\tr(\dot{\AAA}_{1,r}e^{-t\AAA_{1,r}^2}(y,y))\;\sim\;\sum_{k=0}^\infty c_{1,k}(r)(y)t^{(k-n-1)/2}
\]
be the asymptotic expansion as $t\to0$. Then picking up the constant term in \eqref{E:varFormula-expansion} implies
\[
\int_X\frac{\p}{\p r}\AS(\DD_r^+)\;=\;\frac{1}{\sqrt{\pi}}\int_{Y_1} c_{1,n}(r)(y)\,dy,
\]
and \eqref{E:varFormula-noncpt-ind} is proved.
\end{proof}

\begin{remark}\label{R:varFormula-noncpt-ind-2}
The arguments in this proof are not necessary in the compact case, as by the APS index theorem, one can directly work on the variation of eta invariant using the spectral definition.
\end{remark}

\subsection{Variation formula of the spectral relative eta invariant}\label{SS:varFormula-sp}

Before giving the variation formula for the relative eta function $\eta(s;\AAA_1,\AAA_0)$, we explain some subtleties in this process. Recall in the proof of \eqref{E:varFormula-cpt-1}, one has
\[
\frac{\p}{\p r}\Tr(D_re^{-tD_r^2})\;=\;\Big(1+2t\frac{\p}{\p t}\Big)\Tr(\dot{D}_re^{-tD_r^2}),
\]
which indicates, under the assumption that $\dim\ker D_r$ is constant, that for $\Re(s)>n$,
\[
\begin{aligned}
&\frac{\p}{\p r}\int_0^Tt^{(s-1)/2}\Tr(D_re^{-tD_r^2})\,dt \\
&\hspace{1cm}=\;
2T^{(s+1)/2}\Tr(\dot{D}_re^{-TD_r^2})
\,-\,s\int_0^T t^{(s-1)/2}\Tr(\dot{D}_re^{-tD_r^2})\,dt.
\end{aligned}
\]
Since the spectrum of the operator $D_r$ satisfies a Weyl's asymptotic formula, as $T\to\infty$, the first term on the right hand side vanishes and the integral on the right hand side is absolutely convergent, hence \eqref{E:varFormula-cpt-1}. 

Now if one considers the relative heat trace $\Tr(\AAA_{1,r}e^{-t\AAA_{1,r}^2}-\AAA_0e^{-t\AAA_0^2})$, the resulting operator would be $\dot{\AAA}_{1,r}e^{-t\AAA_{1,r}^2}$ and it should satisfy similar properties as discussed just above. Thanks to the fact that $\AAA_{1,r}$ coinciding with $\AAA_0$ at infinity implies $\dot{\AAA}_{1,r}$ vanishing outside a compact set, the desired properties can be established through the following lemmas.

\begin{lemma}\label{L:varTraceclass}
Let $\AAA$ be a self-adjoint Dirac-type operator on $(Y,E)$, where $E$ is a Dirac bundle over a complete manifold $Y$ (without boundary). Suppose $(Y,E,\AAA)$ satisfies Assumption \ref{A:traceclass}. Let $P$ be a pseudo-differential operator of non-negative order on $(Y,E)$ which is identically zero outside a compact set. Then $Pe^{-t\AAA^2}$ is a trace-class operator for $t>0$.
\end{lemma}

\begin{proof}
One can follow the arguments in \cite[Section 3]{Bunke92}. To be precise, by \cite{ChengLiYau81} and \cite{DonnellyLi82}, when $(Y,E,\AAA)$ satisfies Assumption \ref{A:traceclass}, for any $t>0$ the kernel of $e^{-t\AAA^2}$ can be estimated by
\begin{equation}\label{E:HeatKernel est}
|\KK(y,z;t)|\;\le\;C_1e^{C_2{\rm dist}(y,K)}e^{-C_3{\rm dist}(y,z)^2},
\end{equation}
where $K$ is a compact subset of $Y$ outside of which $P$ is vanishing. Let $\MM$ be the multiplication operator by $e^{-\epsilon{\rm dist}(y,K)^2}$ for $\epsilon>0$ small enough and write $Pe^{-t\AAA^2}$ as
\[
Pe^{-t\AAA^2}\;=\;\big(Pe^{-t\AAA^2/2}\MM^{-1}\big)\cdot\big(\MM e^{-t\AAA^2/2}\big).
\]
For any fixed $t>0$, the first factor on the right hand side is Hilbert--Schmidt by noticing that
\[
\begin{aligned}
& \int_K\int_Y\big|P\KK(y,z;t/2)e^{\epsilon{\rm dist}(z,K)^2}\big|^2dzdy \\
& \le\;\sum_{l=1}^\infty e^{2\epsilon l^2}\int_K\int_{U_l\setminus U_{l-1}}\big|P\KK(y,z;t/2)\big|^2dzdy
\end{aligned}
\]
where $U_l:=\{y\in Y:{\rm dist}(y,K)\le l\}$, and a finite propagation speed argument (cf. Lemma \ref{L:STest}); while the second factor is Hilbert--Schmidt by the heat kernel estimate \eqref{E:HeatKernel est} and the assumption on the geometry of the manifold. Therefore $Pe^{-t\AAA^2}$ is a trace-class operator.
\end{proof}

\begin{lemma}\label{L:exp dacay}
Let $Pe^{-t\AAA^2}$ be an operator as in Lemma \ref{L:varTraceclass}. Suppose $\AAA$ is a strictly invertible operator (meaning the spectrum of $\AAA^2$ has a strictly positive lower bound). Then $\Tr(Pe^{-t\AAA^2})$ decays exponentially as $t\to\infty$.
\end{lemma}

\begin{proof}
For $t\ge2$, write $Pe^{-t\AAA^2}=Pe^{-\AAA^2}\cdot e^{-(t-1)\AAA^2}$. The first factor is a trace-class operator while the second factor is a bounded operator. By the assumption on $\AAA$,
\[
\big\Vert Pe^{-t\AAA^2}\big\Vert_1\;\le\;\big\Vert Pe^{-\AAA^2}\big\Vert_1\cdot\big\Vert e^{-(t-1)\AAA^2}\big\Vert\;\le\;e^{-(t-1)\delta}\big\Vert Pe^{-\AAA^2}\big\Vert_1,
\]
for some $\delta>0$, where $\Vert\cdot\Vert_1$ denotes the trace norm and $\Vert\cdot\Vert$ denotes the operator norm.
\end{proof}

With these lemmas, we are ready to show the following variation formula for the spectral relative eta invariant.

\begin{theorem}\label{T:varFormula-noncpt-sp}
Let $\AAA_0$ be a self-adjoint Dirac-type operator on $(Y_0,E_0)$ which is invertible at infinity. Let $\AAA_{1,r}$ be a smooth family of self-adjoint Dirac-type operators on $(Y_1,E_1)$ which coincides with $\AAA_0$ at infinity. Suppose $\AAA_0$ and $\AAA_{1,r}$ satisfy Assumption \ref{A:traceclass}. If there exists the following asymptotic expansion
\begin{equation}\label{E:varSTexpan-A}
\Tr\big(\dot{\AAA}_{1,r}e^{-t\AAA_{1,r}^2}\big)\;\sim\;\sum_{k=0}^\infty c_k(r)t^{(k-n-1)/2},\qquad\text{as }t\to0,
\end{equation}
where $n=\dim Y_0=\dim Y_1$. Then the mod $\ZZ$ reduction $\oeta(0;\AAA_{1,r},\AAA_0)$ of the spectral relative eta invariant is a smooth function of $r$ and
\begin{equation}\label{E:varFormula-noncpt-sp}
\frac{d}{dr}\oeta(0;\AAA_{1,r},\AAA_0)\;=\;-\frac{2}{\sqrt{\pi}}c_n(r).
\end{equation}

In particular, if $\AAA_{1,r}$ is strictly invertible, then \eqref{E:varFormula-noncpt-sp} also holds with $\oeta(0;\AAA_{1,r},\AAA_0)$ replaced by $\eta(0;\AAA_{1,r},\AAA_0)$.
\end{theorem}

\begin{proof}
To simplify notations, we denote $\AAA_{1,r}$ by $\AAA_r$. We consider $\eta(s;\AAA_r,\AAA_0)$ near $r=r_0$. Then there exists a uniform positive lower bound for the essential spectra of $\AAA_r$. Choose $\delta>0$ less than this lower bound such that $\pm\delta\notin{\rm spec}(\AAA_{r_0})$. By continuity, there exists $\epsilon>0$ such that $\pm\delta\notin{\rm spec}(\AAA_r)$ for $r\in[r_0-\epsilon,r_0+\epsilon]$. 
Let $\Pi_r^\delta$ be the orthogonal projection of $L^2(Y_1,E_1)$ onto the subspace spanned by all eigensections of $\AAA_r$ with eigenvalue in $(-\delta,\delta)$. It is a finite-rank operator. For $r\in[r_0-\epsilon,r_0+\epsilon]$, define
\begin{equation}\label{E:Ar'}
\AAA_r'\;:=\;\AAA_r\,(\id\,-\,\Pi_r^\delta)\;+\;\Pi_r^\delta.
\end{equation}
Then $\AAA_r'$ is a strictly invertible operator and depends smoothly on $r$. Note that $\AAA_r'$ differs from $\AAA_r$ only on a finite-dimensional space. So $\eta(s;\AAA_r',\AAA_r)$ is well-defined and
\[
\eta(s;\AAA_r',\AAA_r)\;=\;\Tr(\Pi_r^\delta)
\;-\;\sum_{\substack{\lambda_r\in{\rm spec}(\AAA_r) \\ 0<|\lambda_r|<\delta}}{\rm sign}(\lambda_r)|\lambda_r|^{-s},
\]
which is an entire function. Therefore
\[
\eta(s;\AAA_r,\AAA_0)\;+\;\Tr(\Pi_r^\delta)
\;-\;\sum_{\substack{\lambda_r\in{\rm spec}(\AAA_r) \\ 0<|\lambda_r|<\delta}}{\rm sign}(\lambda_r)|\lambda_r|^{-s}
\]
is absolutely convergent for $\Re(s)>n$. This expression is formally $\eta(s;\AAA_r',\AAA_0)$. Therefore $\eta(s;\AAA_r',\AAA_0)$ is well-defined for $\Re(s)>n$ and
\begin{equation}\label{E:modA'=modA}
\oeta(0;\AAA_r',\AAA_0)\;=\;\oeta(0;\AAA_r,\AAA_0).
\end{equation}

Note that
\begin{align*}
\dot{\AAA}'_r\;&=\;\dot{\AAA}_r\;+\;\frac{d}{dr}((\id-\AAA_r)\Pi_r^\delta) \\
&=\;P_r\;+\;\{\mbox{a finite-rank operator}\},
\end{align*}
where $P_r$ vanishes outside a compact set. By Lemma \ref{L:varTraceclass}, both $\dot{\AAA}_re^{-t\AAA_r^2}$ and $\dot{\AAA}'_re^{-t(\AAA'_r)^2}$ are trace-class operators, and
\[
\Tr\big(\dot{\AAA}'_re^{-t(\AAA_r')^2}\big)\;=\;\Tr\big(\dot{\AAA}_re^{-t\AAA_r^2}\big)\;+\;O(1),\qquad\mbox{as }t\to 0.
\]
Let
\begin{equation}\label{E:varSTexpan-A'}
\Tr\big(\dot{\AAA}'_re^{-t(\AAA'_r)^2}\big)\;\sim\;\sum_{k=0}^\infty c'_k(r)\,t^{(k-n-1)/2},\qquad\mbox{as }t\to 0
\end{equation}
be the short time asymptotic expansion. Then $c'_k(r)=c_k(r)$ for $k\ne n+1$.
By Lemma \ref{L:exp dacay}, $|\Tr(\dot{\AAA}'_re^{-t(\AAA'_r)^2})|$ decays exponentially as $t\to\infty$. With these properties, we are now able to apply \cite[Proposition 2.6]{Muller94} to obtain that $\eta(s;\AAA'_r,\AAA_0)$ depends smoothly on $r$ and
\begin{equation}\label{E:variation}
\frac{\p}{\p r}\eta(s;\AAA'_r,\AAA_0)\;=\;-\,\frac{s}{\Gamma((s+1)/2)}\int_0^\infty t^{(s-1)/2}\Tr\big(\dot{\AAA}'_re^{-t(\AAA'_r)^2}\big)\,dt
\end{equation}
for $\Re(s)>n$.

Using \eqref{E:varSTexpan-A'}, the integral on the right hand side of \eqref{E:variation} admits a meromorphic continuation to the complex plane such that $s=0$ is a simple pole with residue $2c'_n(r)=2c_n(r)$. This means that $\frac{\p}{\p r}\eta(s;\AAA_r',\AAA_0)$ is holomorphic at $s=0$ and by \eqref{E:modA'=modA}
\[
\frac{d}{dr}\oeta(0;\AAA_r,\AAA_0)\;=\;\frac{d}{dr}\eta(s;\AAA_r',\AAA_0)|_{s=0}\;=\;-\frac{2}{\sqrt{\pi}}c_n(r).
\]
Thus \eqref{E:varFormula-noncpt-sp} is proved. When $\AAA_r$ is invertible, one can just put $\AAA_r'$ to be $\AAA_r$, hence $\oeta(0;\AAA_r,\AAA_0)$ can be replaced by $\eta(0;\AAA_r,\AAA_0)$ in \eqref{E:varFormula-noncpt-sp}.
\end{proof}

\begin{remark}\label{R:varFormula-noncpt-sp}
(1) Since $\AAA_r'$ is invertible, the reduced relative eta function $\xi(0;\AAA_r',\AAA_0)$ again depends smoothly on $r$ whose variation is one half of that of $\eta(0;\AAA_r',\AAA_0)$ \eqref{E:variation}. Let $\oxi$ denote the mod $\ZZ$ reduction of the reduced relative eta invariant. From the proof, $\oxi(0;\AAA_r',\AAA_0)=\oxi(0;\AAA_r,\AAA_0)$. Hence one also has the variation formula for $\oxi(0;\AAA_{1,r},\AAA_0)$
\[
\frac{d}{dr}\oxi(0;\AAA_{1,r},\AAA_0)\;=\;-\frac{1}{\sqrt{\pi}}c_n(r).
\]

(2) More generally, suppose the asymptotic expansion of $\dot{\AAA}_{1,r}e^{-t\AAA_{1,r}^2}$ as $t\to0$ has the form
\begin{equation}\label{E:varSTexp-general}
\Tr\big(\dot{\AAA}_{1,r}e^{-t\AAA_{1,r}^2}\big)\;\sim\;\sum_{\substack{\Re(\alpha)\to\infty \\ 0\le k\le k(\alpha)}}c_{\alpha k}(r)t^\alpha\log^kt
\end{equation}
similar to \eqref{E:STexp-general}. Then formula \eqref{E:variation} is still true for $\Re(s)>-2\inf\{\Re(\alpha):c_{\alpha k}(r)\ne0\text{ for some }0\le k\le k(\alpha)\}$. In this case the variation formula \eqref{E:varFormula-noncpt-sp} becomes
\[
\frac{d}{dr}\oeta(0;\AAA_{1,r},\AAA_0)\;=\;-\frac{2}{\sqrt{\pi}}c_{-1/2,0}(r).
\]
And
\[
\frac{d}{dr}\oxi(0;\AAA_{1,r},\AAA_0)\;=\;-\frac{1}{\sqrt{\pi}}c_{-1/2,0}(r).
\]
\end{remark}

\subsection{A formula in terms of spectral flow}\label{SS:reletafunc-sf}

Recall in Corollary \ref{C:sp flow}, if $Y_0=Y_1$ and $\AAA^0=\AAA_0$, then \eqref{E:sp flow-1} becomes
\begin{equation}\label{E:reletainv-sf}
\xi(\AAA_1,\AAA_0)\;- \;\frac{1}{2}\int_0^1\,\Big(\frac{d}{dr}\oeta(\AAA_r,\AAA_0)\Big)\,dr
	\;= \;\spf(\AA).
\end{equation}
In other words, the spectral flow measures the integer jumps of the index-theoretic relative eta invariant for a smooth family of operators. The following proposition says that the spectral relative eta invariant satisfies the same relation.

\begin{proposition}\label{P:reletafunc-sf}
Let
\(\AA  =  \big\{
	  \AAA_r:\,C^\infty(Y,E)\to C^\infty(Y,E)\big\}_{0\le r\le 1}  
\)
be a smooth family of self-adjoint strongly Callias-type operators satisfying Assumption \ref{A:traceclass}. Let $\eta(0;\AAA_r,\AAA_0)$ be defined by Definition \ref{D:releta-sp} and $\xi(0;\AAA_r,\AAA_0)$ be defined by \eqref{E:reduced releta sp}. Then the mod $\ZZ$ reduction $\oeta(0;\AAA_r,\AAA_0)\in \RR/\ZZ$ depends smoothly on $r\in[0,1]$ and
\begin{equation}\label{E:reletafunc-sf}
\xi(0;\AAA_1,\AAA_0)\;-\;\frac{1}{2}\,\int_0^1\Big(\frac{d}{dr}\oeta(0;\AAA_r,\AAA_0)\Big)\,dr
\;=\;\spf(\AA).
\end{equation}
\end{proposition}

\begin{remark}\label{R:varpotential}
Note that under the hypothesis of Proposition \ref{P:reletafunc-sf}, $\AAA_r=A+\Psi_r$, where $A$ is the (compatible) Dirac operator on $(Y,E)$ and $\Psi_r$ is the Callias potential which varies with respect to $r\in[0,1]$.
\end{remark}

\begin{proof}
Due to reason explained in Remark \ref{R:reletafuc}, the proof of the classical case on closed manifolds cannot be verbatim duplicated here. But we can do a little adjustment inspired by Theorem \ref{T:varFormula-noncpt-sp}.
Choose a subdivision $0=r_0<r_1<\cdots<r_m=1$ of the interval $[0,1]$ such that there exist $\delta_i>0$, $i=1,\dots,m$ with $\pm\delta_i\notin{\rm spec}(\AAA_r)$ for $r\in[r_{i-1},r_i]$. Restricted to $[r_{i-1},r_i]$, let $\AAA_r'$ be defined by \eqref{E:Ar'}. Then
\[
\begin{aligned}
\eta(0;\AAA_r',\AAA_{r_{i-1}}')\;=\;\eta(0;\AAA_r,\AAA_{r_{i-1}})\; & +\;\Big(\Tr(\Pi_r^\delta)
\;-\;\sum_{\substack{\lambda_r\in{\rm spec}(\AAA_r) \\ 0<|\lambda_r|<\delta_i}}{\rm sign}(\lambda_r)\Big) \\
& -\;\Big(\Tr(\Pi_{r_{i-1}}^\delta)
\;-\;\sum_{\substack{\lambda_{r_{i-1}}\in{\rm spec}(\AAA_{r_{i-1}}) \\ 0<|\lambda_{r_{i-1}}|<\delta_i}}{\rm sign}(\lambda_{r_{i-1}})\Big).
\end{aligned}
\]
Note that $\Pi_r^\delta$ does not depend on $r$ over $[r_{i-1},r_i]$. So
\begin{equation}\label{E:reletafun<}
\eta(0;\AAA_r,\AAA_{r_{i-1}})\,-\,\eta(0;\AAA_r',\AAA_{r_{i-1}}')\;=\;2\spf(\{\AAA_r\})_{r\in[r_{i-1},r_i]}\,+\,\dim\ker\AAA_{r_{i-1}}\,-\,\dim\ker\AAA_r.
\end{equation}

From the proof of Theorem \ref{T:varFormula-noncpt-sp}, $\eta(0;\AAA_r',\AAA_{r_{i-1}}')=\oeta(0;\AAA_r,\AAA_{r_{i-1}})$ varies smoothly with respect to $r\in[r_{i-1},r_i]$. Hence
\begin{align*}
\int_{r_{i-1}}^{r_i}\Big(\frac{d}{dr}&\oeta(0;\AAA_r,\AAA_{r_{i-1}})\Big)\,dr \\
=\ \ &\eta(0;\AAA_{r_i}',\AAA_{r_{i-1}}')\,-\,\eta(0;\AAA_{r_{i-1}}',\AAA_{r_{i-1}}') \\
\stackrel{\eqref{E:reletafun<}}{=}\;&\eta(0;\AAA_{r_i},\AAA_{r_{i-1}})\,-\,2\spf(\{\AAA_r\})_{r\in[r_{i-1},r_i]}\,+\,\dim\ker\AAA_{r_i}\,-\,\dim\ker\AAA_{r_{i-1}} \\
=\ \ &2\,\xi(0;\AAA_{r_i},\AAA_{r_{i-1}})\,-\,2\spf(\{\AAA_r\})_{r\in[r_{i-1},r_i]}.
\end{align*}
Therefore we conclude
\[
\xi(0;\AAA_{r_i},\AAA_{r_{i-1}})\;-\;\frac{1}{2}\,\int_{r_{i-1}}^{r_i}\Big(\frac{d}{dr}\oeta(0;\AAA_r,\AAA_{r_{i-1}})\Big)\,dr
\;=\;\spf(\{\AAA_r\})_{r\in[r_{i-1},r_i]}.
\]
Summing over all subintervals $[r_{i-1},r_i]$ ($i=1,\dots,m$) gives \eqref{E:reletafunc-sf}
\end{proof}

One can use \eqref{E:reletafunc-sf} to compute the relative eta invariant, as illustrated in the following simple example.

\begin{example}\label{Ex:compute releta}
Let $A=-i\sigma_1\frac{\p}{\p x_1}-i\sigma_2\frac{\p}{\p x_2}$ be the Dirac operator over $\RR^2$ acting on $C^\infty(\RR^2,\CC\oplus\CC)$, where
\[
\sigma_1\;=\;\left(
\begin{matrix}
0 & 1 \\
1 & 0
\end{matrix}
\right),\qquad
\sigma_2\;=\;\left(
\begin{matrix}
0 & -i \\
i & 0
\end{matrix}
\right).
\]
Let $f_0$ is a real-valued smooth function on $\RR^2$ and
\[
F_0\;:=\;\left(
\begin{matrix}
f_0 & 0 \\
0 & -f_0
\end{matrix}
\right)\;:\;
C^\infty(\RR^2,\CC\oplus\CC)\,\to\,C^\infty(\RR^2,\CC\oplus\CC).
\]
Then $\AAA_0=A+F_0$ is a self-adjoint differential operator and
\[
\AAA_0^2\;=\;-\,\frac{\p^2}{\p x_1^2}\,-\,\frac{\p^2}{\p x_2^2}\,+\,\left(
\begin{matrix}
0 & i\frac{\p f_0}{\p x_1}+\frac{\p f_0}{\p x_2} \\
-i\frac{\p f_0}{\p x_1}+\frac{\p f_0}{\p x_2} & 0
\end{matrix}
\right)\,+\,f_0^2
\]
is a Schr\"odinger operator. Assume that $f_0^2$ grows as $\ln|x|$ when $|x|$ is large. Then $\AAA_0^2$ has discrete spectrum and by \cite{Levend96}, the spectral counting function of $\AAA_0^2$ grows exponentially. Hence the eta function \eqref{E:etafunc} could not be defined for $\AAA_0$, let alone the eta invariant.

Let $\AAA_1=\AAA_0+F$ be a compact perturbation of $\AAA_0$ such that
\[
F\;=\;\left(
\begin{matrix}
f & 0 \\
0 & -f
\end{matrix}
\right),
\]
where $f$ is a real-valued smooth function with compact support. Then the eigenvalues of $\AAA_1^2$ have similar asymptotic properties as that of $\AAA_0^2$. Note that $\AAA_0$ and $\AAA_1$ satisfy Assumption \ref{A:traceclass}, so the relative eta invariant $\eta(0;\AAA_1,\AAA_0)$ can be defined.

For $0\le r\le1$, let $\AAA_r:=\AAA_0+rF$. Then $\{\AAA_r\}_{0\le r\le1}$ is a smooth family of self-adjoint unbounded operators on $L^2(\RR^2,\CC\oplus\CC)$ and it has a well-defined spectral flow $\spf(\AA)$ (cf. \cite{Lesch05sf}). On the other hand, $\dot{\AAA_r}=F$. So
\[
\Tr\big(\dot{\AAA}_re^{-t\AAA_r^2}\big)\;=\;\int_{\RR^2}\tr\big(F(x)e^{-t\AAA_r^2}(x,x)\big)\,dx\;=\;\int_{\supp f}\tr\big(Fe^{-t\tilA_r^2}(x,x)\big)\,dx,
\]
where $\tilA_r$ is an extension of $\AAA_r$ to a closed manifold that contains a neighborhood of $\supp f$. One then has
\[
\tr\big(Fe^{-t\tilA_r^2}(x,x)\big)\;\sim\;\sum_{k=0}^\infty c_k(r)(x)t^{(k-3)/2}\qquad\text{as }t\to0.
\]
Since $F$ is a zeroth-order differential operator, by \cite[Lemma 1.7.7]{Gilkey95book}, $c_k(r)(x)=0$ if $k$ is even. It follows from Theorem \ref{T:varFormula-noncpt-sp} that
\[
\frac{d}{dr}\oeta(0;\AAA_r,\AAA_0)\;=\;-\frac{2}{\sqrt{\pi}}\int_{\supp f}c_2(r)(x)\,dx\;=\;0.
\]
Now by \eqref{E:reletafunc-sf},
\[
\eta(0;\AAA_1,\AAA_0)\;=\;2\spf(\AA)\,-\,\dim\ker\AAA_1\,+\,\dim\ker\AAA_0.
\]

Note that $\AAA_0$ and $\AAA_1$ are indeed strongly Callias-type operators. One can construct an almost compact cobordism $(\RR^2\times[0,1],\DD)$ between them as follows. Let
\[
\DD^+\;:=\;\left(
\begin{matrix}
-1 & 0 \\
0 & 1
\end{matrix}
\right)\left(\frac{\p}{\p u}+\AAA_{\rho(u)}\right)\;:\;
C^\infty(\RR^2\times[0,1],\CC\oplus\CC)\,\to\,C^\infty(\RR^2\times[0,1],\CC\oplus\CC),
\]
where $\rho:[0,1]\to[0,1]$ is a smooth non-decreasing function such that $\rho(u)=0$ for $u\le1/3$ and $\rho(u)=1$ for $u\ge2/3$. Then the restriction of $\DD^+$ to $\RR^2\times\{0\}$ is $\AAA_0$ and to $\RR^2\times\{1\}$ is $-\AAA_1$. Now let $E=E^+\oplus E^-:=(\CC\oplus\CC)\oplus(\CC\oplus\CC)$ and
\[
\DD\;:=\;\left(
\begin{matrix}
0 & (\DD^+)^* \\
\DD^+ & 0
\end{matrix}
\right)\;:\;
C^\infty(\RR^2\times[0,1],E)\,\to\,C^\infty(\RR^2\times[0,1],E).
\]
Then $\DD$ is a self-adjoint strongly Callias-type operator, and we see that $\AAA_0$ and $\AAA_1$ are cobordant. Actually $\AAA_0$ and $\AAA_r$ are cobordant for any $0\le r\le1$. Note that since $\RR^2\times[0,1]$ is odd-dimensional, $\oeta(\AAA_r,\AAA_0)\equiv0$. Hence by \eqref{E:reletainv-sf},
\[
\eta(\AAA_1,\AAA_0)\;=\;2\spf(\AA)\,-\,\dim\ker\AAA_1\,+\,\dim\ker\AAA_0.
\]
This shows that $\eta(\AAA_1,\AAA_0)=\eta(0;\AAA_1,\AAA_0)$, which is a special instance of the general result stated in Theorem \ref{T:releta: ind=sp} of the next subsection.
\end{example}

\subsection{A spectral interpretation of the relative eta invariant}\label{SS:releta: ind=sp}

In this subsection, we combine the results obtained in the previous subsections to give an equality between the index-theoretic and spectral relative eta invariants in the case that the operators act on the same domain, which is part (\romannumeral1) of Theorem \ref{T:intro-2}.

\begin{theorem}\label{T:releta: ind=sp}
Suppose $\AAA_0$ and $\AAA_1$ are cobordant self-adjoint strongly Callias-type operators over a complete Riemannian manifold $Y$ acting on sections of a Dirac bundle $E$ and satisfy Assumption \ref{A:traceclass}. Then
\[
\eta(\AAA_1,\AAA_0)\;=\;\eta(0;\AAA_1,\AAA_0).
\]
In particular, if $Y$ is an even-dimensional manifold, then the spectral relative eta invariant $\eta(0;\AAA_1,\AAA_0)$ is an integer.
\end{theorem}

\begin{proof}
For $0\le r\le1$, let $\AAA_r=r\AAA_1+(1-r)\AAA_0$. Then $\AA=\{\AAA_r\}_{0\le r\le1}$ is a smooth family of self-adjoint strongly Callias-type operators on $Y$. In view of \eqref{E:reletainv-sf} and
\eqref{E:reletafunc-sf}, the thesis is reduced to proving
\[
\frac{d}{dr}\oeta(\AAA_r,\AAA_0)\;=\;\frac{d}{dr}\oeta(0;\AAA_r,\AAA_0),
\]
which follows immediately from Theorems \ref{T:varFormula-noncpt-ind} and \ref{T:varFormula-noncpt-sp}.
\end{proof}

\section{Gluing law for the relative eta invariant}\label{S:gluing law}

The gluing problem for eta invariant on compact manifolds has been studied extensively by many authors using different methods, including Wojciechowski \cites{Woj94,Woj95}, Bunke \cite{Bunke95eta} (see also Dai and Freed \cite{DaiFreed94}), M\"uller \cite{Mueller96}, Br\"uning and Lesch \cite{BruningLesch99}, Kirk and Lesch \cite{KirkLesch04}, Loya and Park \cite{LoyaPark06}, etc. Roughly speaking, let $D$ be a self-adjoint Dirac-type operator on a closed manifold $M$ and let $\Sigma$ be a closed hypersurface of $M$ which partitions $M$ into two components $M'$ and $M''$ with common boundary $\Sigma$. Assume $D$ is product in a tubular neighborhood of $\Sigma$. Denote by $D'$ and $D''$ the restrictions of $D$ to the two manifolds with boundary $M'$ and $M''$, respectively, and by $B$ the restriction of $D$ to $\Sigma$. Let $L'$ and $L''$ be two Lagrangian subspaces of $\ker B$. Then one can impose an APS type boundary condition corresponding to $L'$ (resp. $L''$) to $D'$ (resp. $D''$) and get a self-adjoint operator $D_{L'}'$ on $M'$ (resp. $D_{L''}''$ on $M''$). The eta invariants $\eta(D_{L'}')$, $\eta(D_{L''}'')$ and reduced eta invariants $\xi(D_{L'}')=(\eta(D_{L'}')+\dim\ker D_{L'}')/2$, $\xi(D_{L''}'')=(\eta(D_{L''}'')+\dim\ker D_{L''}'')/2$ can be defined as usual. Then the following gluing formula holds
\[
\xi(D)\;=\;\xi(D_{L'}')\;+\;\xi(D_{L''}'')\;+\;m(L',L'')\mod\ZZ,
\]
where $m(L',L'')$ is a real valued function determined by the pair of Lagrangian subspaces $(L',L'')$. In some papers, the integer contribution $\xi(D)-(\xi(D_{L'}')+\xi(D_{L''}'')+m(L',L''))$ is also given (in different expressions). Note that in the special case where $\ker B=0$, i.e., $B$ is invertible, the self-adjoint boundary conditions for $D'$ and $D''$ are just the APS boundary conditions. So the formula is simplified to
\[
\xi(D)\;=\;\xi(D_{\aps}')\;+\;\xi(D_{\aps}'')\mod\ZZ.
\]

In this section, we deduce similar gluing formulas for the index-theoretic and spectral relative eta invariants and show a mod $2\ZZ$ equality between them. Proof of the former formula (Theorem \ref{T:gluing-ind}) is pretty straightforward. While proof of the latter formula (Theorem \ref{T:gluing-sp}) uses the variation formula of Section \ref{S:varFormula} and the method of \cite{BruningLesch99} (adapted to non-compact situation) as the main tools.

\subsection{Gluing formula for $\eta(\AAA_1,\AAA_0)$}\label{SS:gluing-ind}

We introduce the basic setting of the gluing problem for relative eta invariants. Let $(Y_0,E_0,\AAA_0)$ and $(Y_1,E_1,\AAA_1)$ be two triples of self-adjoint strongly Callias-type operators which are cobordant and satisfy Assumption \ref{A:traceclass}. Let $\Sigma_0\cong\Sigma_1\cong\Sigma$ be a common closed hypersurface of $Y_0$ and $Y_1$ which induces the following partitions
\[
Y_0\;=\;Y_0'\cup_{\Sigma_0}Y_0'',\qquad Y_1\;=\;Y_1'\cup_{\Sigma_1}Y_1'',
\]
where $Y_0'$ and $Y_1'$ are compact subsets and $\AAA_0$ and $\AAA_1$ coincide in a neighborhood of $Y_0''\cong Y_1''\cong Y''$. To simplify notations, when we talk about data on the domain where $\AAA_0$ and $\AAA_1$ coincide, we will omit the subscript ``0'' or ``1''.
In the discussion below, we usually choose $Y_0'$ and $Y_1'$ to be large enough. In particular, one can require the restriction of $\AAA_j$ ($j=0,1$) to $\Sigma$, denoted by $\BBB$, to be an invertible operator. Here the restriction is with respect to the inward unit normal vector to the boundary of $Y''$. For simplicity, we assume product structures in a tubular neighborhood $N(\Sigma)\cong[-1,1]\times\Sigma$ of $\Sigma$. For $j=0,1$, let $\AAA_j'$ be $\AAA_j$ restricted to $Y_j'$ and $\AAA_{j,\aps}'$ be $\AAA_j'$ with APS boundary condition. Then $\AAA_{j,\aps}'$ is a self-adjoint operator. With this setting, we have the following gluing law for the index-theoretic relative eta invariant.

\begin{theorem}\label{T:gluing-ind}
$\xi(\AAA_1,\AAA_0)=\xi(\AAA_{1,\aps}')-\xi(\AAA_{0,\aps}')\mod\ZZ$.
\end{theorem}

\begin{proof}
Let $(X,\DD)$ be an almost compact cobordism between $\AAA_0$ and $\AAA_1$. By definition, there exists a compact $K\Subset X$, such that $X\setminus K=Y''\times[0,\varepsilon]$. Since $Y_0'$ and $Y_1'$ are large enough, the hypersurface $\Sigma$ can be assumed to be contained in $Y''$. By \eqref{E:reduced releta-2},
\begin{equation}\label{E:xi=AS}
\xi(\AAA_1,\AAA_0)\;=\;-\int_X\AS(\DD^+)\mod\ZZ,
\end{equation}
where $\AS(\DD^+)$ is supported in $K$. Since $Y_0'$ and $Y_1'$ are isometric near $\Sigma$, one can enclose them by a common compact subset $\bar{Y}$. In other words, one can form the following closed manifolds
\[
\hat{Y}_0\;=\;Y_0'\cup_\Sigma\bar{Y},\qquad
\hat{Y}_1\;=\;Y_1'\cup_\Sigma\bar{Y}.
\]
Then $\hat{X}:=K\cup_{\Sigma\times[0,\varepsilon]}(\bar{Y}\times[0,\varepsilon])$ is a compact manifolds with boundary components $\hat{Y}_0$ and $\hat{Y}_1$. Let $\hat{\DD}$ be the extension of $\DD$ to $\hat{X}$. (This is always possible since one can choose $\bar{Y}$ to be $-Y_0'$, for instance.) Then $\hat{\DD}$ is a product on $\bar{Y}\times[0,\varepsilon]$. Applying the classical APS index theorem \cite{APS1} to $\hat{\DD}^+$, we get
\[
\xi(\hat{\AAA}_1)\,-\,\xi(\hat{\AAA}_0)\;=\;\ind\hat{\DD}_\aps^+\,-\,\int_{\hat{X}}\AS(\hat{\DD}^+)\,+\,\dim\ker\hat{\AAA}_1,
\]
where $\hat{\AAA}_0$ and $-\hat{\AAA}_1$ are the restrictions of $\hat{\DD}^+$ to $\hat{Y}_0$ and $\hat{Y}_1$, respectively. In particular,
\[
\xi(\hat{\AAA}_1)\,-\,\xi(\hat{\AAA}_0)\;=\;-\int_{\hat{X}}\AS(\hat{\DD}^+)\mod\ZZ.
\]
Note that $\AS(\hat{\DD}^+)$ is also supported in $K$ and is actually equal to $\AS(\DD^+)$. Combined with \eqref{E:xi=AS}, one has
\[
\xi(\AAA_1,\AAA_0)\;=\;\xi(\hat{\AAA}_1)\;-\;\xi(\hat{\AAA}_0)\mod\ZZ.
\]
Since $\hat{\AAA}_0$ and $\hat{\AAA}_1$ coincide on $\bar{Y}$, by the gluing formulas for $\xi(\hat{\AAA}_0)$ and $\xi(\hat{\AAA}_1)$,
\[
\xi(\hat{\AAA}_1)\,-\,\xi(\hat{\AAA}_0)\;=\;\xi(\AAA_{1,\aps}')\,-\,\xi(\AAA_{0,\aps}')\mod\ZZ.
\]
Therefore the theorem follows.
\end{proof}

\subsection{A family of boundary conditions interpolating transmission and APS}\label{SS:family bc}

To study the gluing problem for the spectral relative eta invariant, one needs to build up a bridge between the quantity on the original manifolds and that on the partitioned ones. From boundary value problem point of view, the former corresponds to the transmission boundary condition while the latter corresponds to the APS boundary condition. If one can find a family of boundary conditions connecting these two, then one would be able to use the variation formula to finish the work. This is exactly the approach of Br\"uning--Lesch \cite{BruningLesch99}. This idea was also used in \cite[Section 8]{BaerBallmann12} and \cite[Section 5]{BrShi17} to prove a splitting theorem of the index.

We now review the setting in \cite{BruningLesch99} accommodating to our situation. The reader can consult  \cite{BruningLesch99} for details. Let $(Y_0,E_0,\AAA_0)$ and $(Y_1,E_1,\AAA_1)$ be two triples of self-adjoint Dirac-type operators satisfying Assumption \ref{A:traceclass}. The closed hypersurface $\Sigma$ is chosen as in Subsection \ref{SS:gluing-ind} except that it does not need to split $Y_j$ into two components.  Denote $E_{N(\Sigma)}:=E|_{N(\Sigma)}$ and $E_\Sigma:=E|_\Sigma$. Note that $L^2(E_{N(\Sigma)})\cong L^2([-1,1],L^2(E_\Sigma))$. Consider the isomorphism
\[
\begin{aligned}
\Phi\,:\,L^2([-1,1],L^2(E_\Sigma))&\;\to\;L^2([0,1],L^2(E_\Sigma)\oplus L^2(E_\Sigma)), \\
s(u)\quad & \;\mapsto\;\quad s(u)\oplus s(-u),\qquad u\in[0,1].
\end{aligned}
\]
Under $\Phi$, $\AAA_j$ is transformed to
\begin{equation}\label{E:boundary transform}
\tilA_j\;:=\;\left(
\begin{matrix}
\rmc(\nu) & 0 \\
0 & -\rmc(\nu)
\end{matrix}
\right)\left(\p_u\,+\,
\left(
\begin{matrix}
\BBB & 0 \\
0 & -\BBB
\end{matrix}
\right)\right)
\;=:\;\tilde{\rmc}(\nu)\left(\p_u\,+\,\tilB\right).
\end{equation}
For $|\theta|<\pi/2$, define a family of boundary conditions for $\tilA_j$ to be
\begin{equation}\label{E:family bc-1}
\cos\theta\,\Pi_+(\tilB)s(0)\;=\;\sin\theta\,\tau\,\Pi_-(\tilB)s(0),\qquad s\in L^2([0,1],L^2(E_\Sigma)\oplus L^2(E_\Sigma)),
\end{equation}
where $\Pi_\pm(\tilB)$ is the spectral projection onto the eigenspaces corresponding to positive (or negative) eigenvalues of $\tilB$ (recall that $\BBB$ can be made invertible) and
\[
\tau\;=\;\left(
\begin{matrix}
0 & 1 \\
1 & 0
\end{matrix}
\right)\,\otimes\,\id_{L^2(E_\Sigma)}.
\]
One can check that $\theta=0$ corresponds to the APS boundary condition and $\theta=\pi/4$ corresponds to the transmission boundary condition. Set
\[
\tilPi(\theta)\;:=\;\cos^2\theta\,\Pi_+(\tilB)\,+\,\sin^2\theta\,\Pi_-(\tilB)\,-\,\frac{1}{2}(\sin 2\theta)\,\tau\,(\Pi_+(\tilB)+\Pi_-(\tilB)).
\]
Then \eqref{E:family bc-1} can even be shortened to
\begin{equation}\label{E:family bc-2}
\tilPi(\theta)s(0)\;=\;0.
\end{equation}
In addition, one has the relation
\[
\tilPi(\theta)\;=\;U(\theta)\tilPi(0)U(\theta)^*,
\]
where
\[
U(\theta)\;:=\;e^{iT(\theta)}\;=\;\cos\theta\,(\Pi_+(\tilB)+\Pi_-(\tilB))\,+\,\sin\theta\,(\Pi_+(\tilB)-\Pi_-(\tilB))\,\tau
\]
is a unitary operator and
\[
T(\theta)\;:=\;-i\,(\Pi_+(\tilB)-\Pi_-(\tilB))\,\tau\theta
\]
is a self-adjoint operator.

One can then use this interpretation to construct a family of operators with varying boundary conditions. To be precise, let $\dom\AAA_j(0)$ denote the domain of the operator $\AAA_j$ with APS boundary condition at $\Sigma$. Choose a function $\varphi\in C_0^\infty((-1,1),\RR)$ such that $\varphi\equiv1$ near 0. Introduce a unitary transformation $\Phi_\theta$ on $L^2([0,1],L^2(E_\Sigma)\oplus L^2(E_\Sigma))$ to be
\[
\Phi_\theta s(u)\;:=\;e^{i\varphi(u)T(\theta)}(s(u)).
\]
Observe that $\tilPi(0)s(0)=0$ implies $\tilPi(\theta)\Phi_\theta s(0)=0$. So $\Phi_\theta$ transfers the boundary condition for $\theta=0$ to the boundary condition for $\theta$. Note that $\Phi$ can be extended to the whole $Y_j$ as the identity outside $N(\Sigma)$, under which $\Phi_\theta$ can be extended in the same manner. Put
\begin{equation}\label{E:Omega}
\Omega_\theta\;:=\;\Phi^*\Phi_\theta\Phi.
\end{equation}
Then one gets a family of domains $\dom\AAA_j(\theta):=\Omega_\theta(\dom\AAA_j(0))$ for $\AAA_j$ over which giving rise to a family of boundary value problems $\AAA_j(\theta)$. Among them $\AAA_j(0)$ is the APS boundary value problem and $\AAA_j(\pi/4)$ is the transmission boundary value problem (which is just $\AAA_j$ on $Y_j$ viewed as a manifold without boundary). For the convenience of employing the variation formula, we set
\[
\AAA_{j,\theta}\;:=\;\Omega_\theta^*\AAA_j(\theta)\Omega_\theta,\qquad |\theta|<\frac{\pi}{2}
\]
to make them a family acting on the same domain $\dom\AAA_j(0)$. By B\"ar--Ballmann's theory of boundary value problems on manifolds with compact boundary \cite{BaerBallmann12} (see also \cite[Sections 4, 5]{BrShi17}), one gets that

\begin{proposition}\label{P:selfadj-bvp}
The operators $\AAA_j(\theta)$ and $\AAA_{j,\theta}$ are self-adjoint with elliptic boundary conditions in the sense of \cite{BaerBallmann12}.
\end{proposition}

\subsection{Gluing formula for $\eta(0;\AAA_1,\AAA_0)$}\label{SS:gluing-sp}

In this subsection, we apply the variation formula Theorem \ref{T:varFormula-noncpt-sp} to $\oeta(0;\AAA_{1,\theta},\AAA_0)$ and combine the result in \cite{BruningLesch99} to deduce a gluing formula for the spectral relative eta invariant. Note that $\AAA_{1,\theta}$ is actually (up to conjugation) a Dirac-type operator on a manifold \emph{with boundary}. The trace class property Theorem \ref{T:traceclass} does not hold any more for $\AAA_{1,\theta}e^{-t\AAA_{1,\theta}^2}-\AAA_0e^{-t\AAA_0^2}$. In this case we need the following assumption.

\begin{assumption}\label{A:traceclass-bvp}
For $j=0,1$, let $\AAA_j$ be a formally self-adjoint Dirac-type operator on $(Y_j,E_j)$ (without boundary) which is invertible at infinity. Assume that
\begin{enumerate}
\item $Y_j$ and $E_j$ have bounded geometry of order $m>n/2$, where $n=\dim Y_j$, namely
\begin{itemize}
\item $Y_j$ has uniformly positive injectivity radius,
\item The curvature tensor of $Y_j$ and its covariant derivatives up to order $m$ are uniformly bounded,
\item The curvature tensor of $E_j$ and its covariant derivatives up to order $m$ are uniformly bounded;
\end{itemize}
\item There exists a constant $C>0$ such that for all $1\le k\le m$ and $s\in\dom(\AAA_j^k)$,
\[
\Vert\AAA_j^ks\Vert_{L^2(Y_j,E_j)}^2\,+\,\Vert s\Vert_{L^2(Y_j,E_j)}^2\;\ge\;C\,\Vert A_j^ks\Vert_{L^2(Y_j,E_j)}^2,
\]
where $A_j$ is the (compatible) Dirac operator on $(Y_j,E_j)$.
\end{enumerate}
\end{assumption}

\begin{remark}\label{R:traceclass-bvp}
It is obvious that this assumption is stronger than Assumption \ref{A:traceclass}. However, note that (\romannumeral2) is weaker than \cite[Assumption 5.1]{FoxHaskell05}. In particular (\romannumeral2) is automatically satisfied when $\AAA_j$ is the spin Dirac operator on a spin manifold whose scalar curvature has a uniformly positive lower bound at infinity.
\end{remark}

Under Assumption \ref{A:traceclass-bvp}, the $k$th Sobolev norm ($1\le k\le m$) on $E_j$ associated to the connection is equivalent to that associated to $A_j$, which is uniformly bounded from above by the norm associated to $\AAA_j$ (cf. \cites{Roe88-index,Bunke93comparison}). We think of both $\AAA_0$ and $\AAA_1(\theta)$ as Dirac-type operators with certain boundary conditions along $\Sigma_0\subset Y_0$ and $\Sigma_1\subset Y_1$, respectively. Since both conditions induce self-adjoint elliptic boundary value problems, from the elliptic regularity of \cite{BaerBallmann12} and the results of \cite{Bunke93comparison}, the above relation between Sobolev norms still hold. Thus for all $t>0$ and $|\theta|<\pi/2$
\[
\AAA_1(\theta)e^{-t\AAA_1^2(\theta)}\,-\,\AAA_0e^{-t\AAA_0^2}
\]
is a trace-class operator. Here both $\AAA_0$ and $\AAA_1(\theta)$ are self-adjoint operators acting on the Hilbert space
\[
\HH\;=\;L^2(K_0,E_0|_{K_0})\,\oplus\,L^2(K_1,E_1|_{K_1})\,\oplus\,L^2(U,E_0|_U),
\]
where $Y_j=K_j\cup U$. It should be pointed out that $K_j\supset N(\Sigma_j)\cong N(\Sigma)$. Then the isometry $\Omega_\theta$ \eqref{E:Omega} associated to $\AAA_1(\theta)$ can be extended to $\HH$ by assigning to be the identity on $L^2(K_0,E_0|_{K_0})$. Using the same notation, one gets a smooth family of trace-class operators
\[
\Omega_\theta^*\,\big(\AAA_1(\theta)e^{-t\AAA_1^2(\theta)}\,-\,\AAA_0e^{-t\AAA_0^2}\big)\,\Omega_\theta\;=\;\AAA_{1,\theta}e^{-t\AAA_{1,\theta}^2}\,-\,\AAA_0e^{-t\AAA_0^2},\qquad|\theta|<\frac{\pi}{2}
\]
with a fixed domain.

We first talk about the well-definedness of the relative eta invariant $\eta(0;\AAA_{1,\theta},\AAA_0)$, which is related to the short and large time asymptotic properties of
\[
\Tr\big(\AAA_1(\theta)e^{-t\AAA_1^2(\theta)}\,-\,\AAA_0e^{-t\AAA_0^2}\big)\;=\;\Tr\big(\tilA_1(\theta)e^{-t\tilA_1^2(\theta)}\,-\,\AAA_0e^{-t\AAA_0^2}\big),
\]
where $\tilA_1(\theta)=\Phi\AAA_1(\theta)\Phi^*$ having the form \eqref{E:boundary transform} near $\Sigma_1$. The large time exponentially decaying property Proposition \ref{P:LTbehav} holds without changes. The main difference is the short time asymptotic expansions, where they may no longer have the form of Proposition \ref{P:STexpan-relative}. As in \cite{BruningLesch99}, we introduce a model operator $\AAA_{1,{\rm mod}}(\theta)$ which has the form \eqref{E:boundary transform} but is defined on $L^2([0,\infty),L^2(E_{\Sigma_1}))$ with boundary condition \eqref{E:family bc-2}. It coincides with $\tilA_1(\theta)$ near $\Sigma_1$. Denote the heat kernel of $\AAA_{1,{\rm mod}}(\theta)$ by $\KK_{1,{\rm mod}}(\theta)$. One can then construct a parametrix $\EE_1(\theta)$ for the kernel of $e^{-t\tilA_1^2(\theta)}$ similar to \eqref{E:parametrix} by patching up three heat kernels. Roughly speaking, $\EE_1(\theta)$ is constructed such that near $\Sigma_1$, it is the heat kernel $\KK_{1,{\rm mod}}(\theta)$; away from $\Sigma_1$, it is the heat kernel of the closed double of $\tilA_1(\theta)$ on $K_1$ and the heat kernel of $\AAA_0$ on $U$. The following is an analogue of Lemma \ref{L:STest}.

\begin{lemma}\label{L:STest-bvp}
Suppose $(Y_1,E_1,\AAA_1)$ satisfies Assumption \ref{A:traceclass-bvp}. For $|\theta|<\pi/2$, let $\tilde{\KK}_1(\theta)(y,z;t)$ be the heat kernel of $\tilA_1(\theta)$ and $\EE_1(\theta)(y,z;t)$ be the parametrix constructed above. Then for $0\le k\le m$, there exist constants $\alpha,\beta>0$ such that as $t\to0$,
\[
\int_{Y_1}\big|\tilA_1^k(\theta)\tilde{\KK}_1(\theta)(y,y;t)-\tilA_1^k(\theta)\EE_1(\theta)(y,y;t)\big|\,dy\;\le\;\alpha e^{-\beta/t}.
\]
\end{lemma}

The proof is essentially the same as that of Lemma \ref{L:STest}. The only extra point is the relation between the Sobolev norms associated to the connection and to the operator $\tilA_1(\theta)$ near the compact boundary $\Sigma_1$, which is indicated by Assumption \ref{A:traceclass-bvp}.

Lemma \ref{L:STest-bvp} implies that one can again replace $\tilde{\KK}_1(\theta)$ by $\EE_1(\theta)$ when considering the short time asymptotic expansion of $\Tr(\AAA_1(\theta)e^{-t\AAA_1^2(\theta)}-\AAA_0e^{-t\AAA_0^2})$. Compared to the case of manifolds without boundary, there is an extra term $\KK_{1,{\rm mod}}(\theta)$ involved in the process. In \cite[Section 4]{BruningLesch99}, an explicit and rather sophisticated formula for the heat kernel $\KK_{1,{\rm mod}}(\theta)$ is given which plays the key role in deriving the variation formulas and gluing law. In particular, it indicates that $\Tr(\AAA_1(\theta)e^{-t\AAA_1^2(\theta)}-\AAA_0e^{-t\AAA_0^2})$ has an asymptotic expansion of the form \eqref{E:STexp-general} as $t\to0$. Then by Remark \ref{R:releta-general}, we have

\begin{proposition}\label{P:releta-sp-bvp}
Under Assumption \ref{A:traceclass-bvp}, the relative eta function $\eta(s;\AAA_{1,\theta},\AAA_0)$ is well-defined when $\Re(s)$ is large and admits a meromorphic continuation to the whole complex plane. Thus one can talk about the relative eta invariant $\eta(0;\AAA_{1,\theta},\AAA_0)$ as defined in Remark \ref{R:releta-general}.
\end{proposition}

This generalizes the definition of spectral relative eta invariant to non-compact manifolds with boundary under self-adjoint elliptic boundary conditions.

We now consider the variation of $\eta(0;\AAA_{1,\theta},\AAA_0)$. From what was presented in Subsection \ref{SS:family bc}, $T'(\theta)$ commutes with $T(\theta)$; both of them commute with $\tilde{\rmc}(\nu)$ and anti-commute with $\tilB$. Hence
\[
\dot{\AAA}_{1,\theta}\;=\;\frac{d}{d\theta}\big[\Phi^*\Phi_\theta^*\,\tilde{\rmc}(\nu)(\p_u+\tilB)\,\Phi_\theta\Phi\big]\;=\;\Phi^*\Phi_\theta^*\,i\tilde{\rmc}(\nu)\big(\varphi'T'(\theta)-2\varphi T'(\theta)\tilB\big)\,\Phi_\theta\Phi,
\]
and is supported in $N(\Sigma_1)$. One can check that under Assumption \ref{A:traceclass-bvp}, the heat kernel of $\tilA_1(\theta)$ has the estimate
\[
|\tilde{\KK}_1(\theta)(y,z;t)|\;\le\;Ce^{-{\rm dist}(y,z)^2/8t},
\]
(cf. \cite{Bunke93comparison}). Then Lemmas \ref{L:varTraceclass} and \ref{L:exp dacay} still hold on manifolds with boundary. (Here the order of the pseudo-differential operator $P$ does not exceed $m$ of Assumption \ref{A:traceclass-bvp}.) Thus $\dot{\AAA}_{1,\theta}e^{-t\AAA_{1,\theta}^2}$ is a trace-class operator and
\[
\Tr\big(\dot{\AAA}_{1,\theta}e^{-t\AAA_{1,\theta}^2}\big)\;=\;\Tr\big[i\tilde{\rmc}(\nu)\big(\varphi'T'(\theta)-2\varphi T'(\theta)\tilB\big)\,e^{-t\tilA_1^2(\theta)}\big].
\]
As mentioned above, the idea is to replace $\tilA_1(\theta)$ by the model operator $\AAA_{1,{\rm mod}}(\theta)$ in the short time asymptotic expansion. To be precise, we have the following lemma.

\begin{lemma}\label{L:varSTest-bvp}
Suppose $(Y_1,E_1,\AAA_1)$ satisfies Assumption \ref{A:traceclass-bvp}. Assume that $P$ is a pseudo-differential operator of order $0\le k\le m$ with compact support $K_P\subset N(\Sigma_1)$. Then there exist constants $\alpha,\beta>0$ such that for $|\theta|<\pi/2$, as $t\to0$,
\[
\int_{K_P}\big|P\tilde{\KK}_1(\theta)(y,y;t)-P\KK_{1,{\rm mod}}(\theta)(y,y;t)\big|\,dy\;\le\;\alpha e^{-\beta/t}.
\]
\end{lemma}

\begin{proof}
Let $\EE_1(\theta)$ be the parametrix of $\tilde{\KK}_1(\theta)$ in Lemma \ref{L:STest-bvp}. Since $\tilA_1(\theta)$ and $\AAA_{1,{\rm mod}}(\theta)$ are equal on $N(\Sigma_1)$, one can require that $\EE_1(\theta)$ is equal to $\KK_{1,{\rm mod}}(\theta)$ on $K_P$. Then by Lemma \ref{L:STest-bvp}, as $t\to0$
\[
\begin{aligned}
& \int_{K_P}\big|\tilA_1^k(\theta)\tilde{\KK}_1(\theta)(y,y;t)-\tilA_1^k(\theta)\EE_1(\theta)(y,y;t)\big|\,dy \\
=\; & \int_{K_P}\big|\tilA_1^k(\theta)\tilde{\KK}_1(\theta)(y,y;t)-\tilA_1^k(\theta)\KK_{1,{\rm mod}}(\theta)(y,y;t)\big|\,dy
\;\le\;\alpha e^{-\beta/t}.
\end{aligned}
\]
The desired estimate then follows from elliptic regularity.
\end{proof}

One can check that under unitary conjugation, the proof of the variation formula Theorem \ref{T:varFormula-noncpt-sp} still works. From Lemma \ref{L:varSTest-bvp}, Theorem \ref{T:varFormula-noncpt-sp} and Remark \ref{R:varFormula-noncpt-sp}, we conclude that

\begin{proposition}
\[
\frac{d}{d\theta}\oxi(0;\AAA_{1,\theta},\AAA_0)\;=\;-\frac{1}{\sqrt{\pi}}c_{-1/2,0}(\theta),
\]
where $c_{-1/2,0}(\theta)$ is the coefficient in the short time asymptotic expansion \eqref{E:varSTexp-general} for
\[
\Tr\big[i\tilde{\rmc}(\nu)\big(\varphi'T'(\theta)-2\varphi T'(\theta)\tilB\big)\,e^{-t\AAA_{1,{\rm mod}}^2(\theta)}\big].
\]
\end{proposition}

Now we recall a crucial result of
\cite{BruningLesch99}.

\begin{proposition}\label{P:gluing-sp}
In the setting of Subsection \ref{SS:family bc}, for $|\theta|<\pi/2$
\[
c_{-1/2,0}(\theta)\;\equiv\;0.
\]
\end{proposition}

The idea of proving this result is using the explicit formula of the heat kernel $\KK_{1,{\rm mod}}(\theta)$ to reduce the coefficient to another coefficient in the short time asymptotic expansion of the kernel of $i\tilde{\rmc}(\nu)T'(\theta)e^{-t\tilB^2}$. In the whole process, several undercover commuting/anti-commuting features associated to the setting of Subsection \ref{SS:family bc} are involved.

From this proposition, one infers that $\oxi(0;\AAA_{1,\theta},\AAA_0)$ is constant for $|\theta|<\pi/2$. Notice that when $\theta=\pi/4$, $\oxi(0;\AAA_{1,\pi/4},\AAA_0)$ is just the mod $\ZZ$ reduction of the reduced relative eta invariant associated to $\AAA_0$ and $\AAA_1$ on two manifolds without boundary; when $\theta=0$, $\oxi(0;\AAA_{1,0},\AAA_0)$ is that associated to $\AAA_0$ and $\AAA_1$ where $\AAA_1$ lives on the manifold $Y_1$ with APS boundary condition near $\Sigma_1$. In the case that $\Sigma_1$ splits $Y_1$, we obtain the following gluing formula for the spectral relative eta invariant.

\begin{theorem}\label{T:gluing-sp}
Let $\AAA_0$ and $\AAA_1$ be two self-adjoint Dirac-type operators on $Y_0$ and $Y_1$ (both without boundary), respectively, which coincide and are invertible at infinity. Assume $\AAA_0$ and $\AAA_1$ satisfy Assumption \ref{A:traceclass-bvp}. Suppose a closed hypersurface $\Sigma_1$ splits $Y_1$ into two submanifolds $Y_1'$ and $Y_1''$ with boundary as described in Subsection \ref{SS:gluing-ind}, where $Y_1'$ is compact. Let $\AAA_1'$ (resp. $\AAA_1''$) be the restriction of $\AAA_1$ to $Y_1'$ (resp. $Y_1''$) and $\AAA_{1,\aps}'$ (resp. $\AAA_{1,\aps}''$) be $\AAA_1'$ (resp. $\AAA_1''$) with APS boundary condition. Then the relative eta invariant $\eta(0;\AAA_{1,\aps}'',\AAA_0)$ is well-defined and
\begin{equation}\label{E:gluing-sp}
\xi(0;\AAA_1,\AAA_0)\;=\;\xi(\AAA_{1,\aps}')\;+\;\xi(0;\AAA_{1,\aps}'',\AAA_0)\mod\ZZ.
\end{equation}
\end{theorem}

\begin{proof}
One does the deformation of boundary conditions to $\AAA_1$ near $\Sigma_1$ to get a family of boundary value problems $\AAA_{1,\theta}$ as in Subsection \ref{SS:family bc}. Then by Proposition \ref{P:gluing-sp}, $\oxi(0;\AAA_{1,\pi/4},\AAA_0)=\oxi(0;\AAA_{1,0},\AAA_0)$. The former one is $\oxi(0;\AAA_1,\AAA_0)$. Under the hypothesis of the theorem, $\AAA_{1,0}=\AAA_{1,\aps}'\oplus\AAA_{1,\aps}''$ lies on two disjoint manifolds $Y_1'$ and $Y_1''$. Since $Y_1'$ is compact so that $\eta(\AAA_{1,\aps}')$ is well-defined, hence $\eta(0;\AAA_{1,\aps}'',\AAA_0)$ is well-defined and $\eta(0;\AAA_{1,\aps}'',\AAA_0)=\eta(0;\AAA_{1,0},\AAA_0)-\eta(\AAA_{1,\aps}')$. Therefore one concludes that
\[
\oxi(0;\AAA_1,\AAA_0)\;=\;\oxi(\AAA_{1,\aps}')\;+\;\oxi(0;\AAA_{1,\aps}'',\AAA_0),
\]
from which \eqref{E:gluing-sp} follows.
\end{proof}

If we think of $Y_0\cup Y_1$ as a whole part, the gluing formula \eqref{E:gluing-sp} can be understood in the following sense. One splits $Y_0\cup Y_1$ into two parts, $Y_0'$ and $Y_0''\cup Y_1$. The former one is a compact set and the eta invariant can be defined individually, while the latter one is non-compact (now with boundary) and the eta invariant can only be defined in the relative sense (since $Y_0''$ and $Y_1$ coincide at infinity). Then the gluing formula says that up to integers the original relative eta invariant is equal to the sum of the (relative) eta invariants on the two new parts.

A particular case is that both manifolds are cut by the same hypersurface which is exactly the case in Subsection \ref{SS:gluing-ind}. In this case the right hand side of \eqref{E:gluing-sp} would contain two individual eta invariants on compact manifolds and a relative eta invariant $\xi(0;\AAA_{1,\aps}'',\AAA_{0,\aps}'')$ which vanishes identically. Thus we deduce

\begin{corollary}\label{C:gluing-sp}
Under the hypothesis of Theorem \ref{T:gluing-sp}. Suppose both $Y_0$ and $Y_1$ are split by the same hypersurface $\Sigma_0\cong\Sigma_1$ lying in the subsets where $\AAA_0$ and $\AAA_1$ coincide. Then
\[
\xi(0;\AAA_1,\AAA_0)\;=\;\xi(\AAA_{1,\aps}')\;-\;\xi(\AAA_{0,\aps}')\mod\ZZ.
\]
\end{corollary}

\subsection{A mod $2\ZZ$ spectral interpretation of the relative eta invariant}\label{SS:releta: ind=sp modZ}

In Subsection \ref{SS:releta: ind=sp}, we obtained an equality between the index-theoretic and spectral relative eta invariants under the circumstance that both operators are on the same manifold. Now we apply the gluing law of preceding subsections to generalize this equality to the case that the operators are on different manifolds. This is part (\romannumeral2) of Theorem \ref{T:intro-2}. The price paid is a stronger assumption and a mod $2\ZZ$ compromise.

\begin{theorem}\label{T:releta: ind=sp mod2Z}
Suppose $\AAA_0$ and $\AAA_1$ are cobordant self-adjoint strongly Callias-type operators on complete Riemannian manifolds $(Y_0,E_0)$ and $(Y_1,E_1)$, respectively satisfying Assumption \ref{A:traceclass-bvp}. Then
\[
\eta(\AAA_1,\AAA_0)\;=\;\eta(0;\AAA_1,\AAA_0)\mod2\ZZ.
\]
\end{theorem}

\begin{proof}
Cut $Y_0$ and $Y_1$ along a hypersurface $\Sigma_0\cong\Sigma_1$ in the subset where $\AAA_0$ and $\AAA_1$ coincide. From Theorem \ref{T:gluing-ind} and Corollary \ref{C:gluing-sp},
\[
\xi(\AAA_1,\AAA_0)\;=\;\xi(0;\AAA_1,\AAA_0)\mod\ZZ.
\]
The theorem then follows by noticing that
\[
\eta(\AAA_1,\AAA_0)\,-\,\eta(0;\AAA_1,\AAA_0)\;=\;2\,\big(\xi(\AAA_1,\AAA_0)\,-\,\xi(0;\AAA_1,\AAA_0)\big).
\]
\end{proof}

Recall that by definition, $\eta(\AAA_1,\AAA_0)$ and $\xi(\AAA_1,\AAA_0)$ are always integers when the manifolds $Y_0$ and $Y_1$ are even-dimensional. So we have

\begin{corollary}\label{C:releta: ind=sp mod2Z}
Suppose $\AAA_0$ and $\AAA_1$ are cobordant self-adjoint strongly Callias-type operators on even-dimensional complete Riemannian manifolds $(Y_0,E_0)$ and $(Y_1,E_1)$, respectively satisfying Assumption \ref{A:traceclass-bvp}. Then the reduced spectral relative eta invariant $\xi(0;\AAA_1,\AAA_0)$ is an integer.

Putting it in another way, the mod $\ZZ$ reduction of the reduced spectral relative eta invariant $\oxi(0;\AAA_1,\AAA_0)\in\RR/\ZZ$ is an almost compact cobordism invariant in the category of self-adjoint strongly Callias-type operators on even-dimensional manifolds, meaning that $\oxi(0;\AAA_1,\AAA_0)=0$ if $\AAA_0$ and $\AAA_1$ are cobordant.
\end{corollary}

\bibliographystyle{amsplain}


\end{document}